\documentclass[12pt,a4paper,leqno]{amsart}

\usepackage[a4paper]{geometry}
\usepackage[pdftex]{graphicx}
\usepackage{layout}
\usepackage{fancyhdr}
\usepackage{setspace}
\usepackage[english]{babel}
\usepackage{mathrsfs,amsmath,amssymb,amstext,amsthm,amsfonts,color,ngerman,bbm,ulem,calc,graphicx,accents,enumitem}
\usepackage[arrow, matrix, curve]{xy}
\usepackage[ansinew]{inputenc}

\usepackage{tikz,xifthen}
\usetikzlibrary{arrows,calc,decorations.pathmorphing,backgrounds,positioning,fit,petri,matrix}

\numberwithin{equation}{section}          

\newtheorem{thm}{Theorem}
\numberwithin{thm}{section}

\newcommand{\rubrik}{}
\newtheorem{prop}[thm]{Proposition}

\newtheorem{lem}[thm]{Lemma}

\theoremstyle{definition}

\newtheorem{defn}[thm]{Definition}
\newtheorem{ex}[thm]{Example}

\theoremstyle{remark}

\newtheorem{rem}[thm]{Remark}     

\theoremstyle{plain}

\DeclareRobustCommand{\SkipTocEntry}[5]{}

\newcommand{\pr}{{\mathrm{pr}}}
\newcommand{\jap}[1] {\langle{#1}\rangle}

\newcommand{\pfe}{\Phi}
\newcommand{\pp}{{\prime\prime}}
\newcommand{\id}{\mathrm{id}}

\newcommand{\BB}{{\mathbb{B}}}
\newcommand{\BBd}{{\BB^d}}
\newcommand{\BBs}{{\BB^s}}
\newcommand{\BBdo}{{\left(\BBd\right)^o}}

\newcommand{\NN}{{\mathbb{N}_0}}
\newcommand{\NNd}{{\mathbb{N}_0^d}}
\newcommand{\NNz}{{\mathbb{N}}}

\newcommand{\RR}{{\mathbb{R}}}
\newcommand{\RRd}{{\RR^d}}
\newcommand{\RRs}{{\RR^s}}
\newcommand{\RRdz}{{\RRd\setminus\{0\}}}

\newcommand{\SSS}{{\mathbb{S}}}
\newcommand{\SSSd}{{\mathbb{S}^{d-1}}}
\newcommand{\ZZ}{{\mathbb{Z}}}

\newcommand{\cl}{{\mathrm{cl}}}
\newcommand{\mm}{{{m_e,m_\psi}}}
\newcommand{\nn}{{{n_e,n_\psi}}}

\newcommand{\SG}{{\mathrm{SG}}}
\newcommand{\SGmm}{{\mathrm{SG}^{\mm}}}
\newcommand{\SGcl}{\mathrm{SG}_\mathrm{cl}}

\newcommand{\SGzz}{{\mathrm{SG}^{0,0}}}

\newcommand{\Sm}{{\mathscr{C}^\infty}}

\newcommand{\Sw}{{\mathscr{S}}}
\newcommand{\SwRRd}{{\mathscr{S}\left(\RRd\right)}}
\newcommand{\Swd}{{\mathscr{S}^\prime}}

\newcommand{\SwdRRd}{{\mathscr{S}^\prime\left(\RRd\right)}}

\newcommand{\cL}{\mathcal{L}} 
\newcommand{\cA}{\mathcal{A}} 

\newcommand{\Op}{\operatorname{Op}}

\newcommand{\WSG}{\mathcal{W}_\SG}
\newcommand{\Wt}{\widetilde{\mathcal{W}}_\SG}
\newcommand{\Wte}{\widetilde{\mathcal{W}}^e_\SG}
\newcommand{\Wtp}{\widetilde{\mathcal{W}}^\psi_\SG}
\newcommand{\Wtpe}{\widetilde{\mathcal{W}}^{\psi e}_\SG}
\newcommand{\WFcl}{{\mathrm{WF}_\mathrm{cl}}}

\newcommand{\WFSGt}{\widetilde{\mathrm{WF}}_\SG}

\newcommand{\lpt}{{\tilde{\lambda}_\varphi}}
\newcommand{\Lp}{\Lambda_\varphi}
\newcommand{\Lpt}{\tilde{\Lambda}_\varphi}
\newcommand{\Lpte}{\tilde{\Lambda}_\varphi^e}
\newcommand{\Lptp}{\tilde{\Lambda}_\varphi^\psi}

\newcommand{\Lt}{\widetilde{\Lambda}}
\newcommand{\Lte}{\widetilde{\Lambda}^e}
\newcommand{\Ltp}{\widetilde{\Lambda}^\psi}
\newcommand{\Ltpe}{\widetilde{\Lambda}^{\psi e}}
\newcommand{\Lpe}{{\Lambda_\varphi^e}}
\newcommand{\Lpp}{{\Lambda_\varphi^\psi}}
\newcommand{\Lppe}{{\Lambda_\varphi^{\psi e}}}
\newcommand{\Lme}{{\Lambda^{e}}}
\newcommand{\Lmp}{{\Lambda^{\psi}}}
\newcommand{\Lmpe}{{\Lambda^{\psi e}}}
\newcommand{\Cp}{\mathcal{C}_\varphi}
\newcommand{\Cpt}{\widetilde{\mathcal{C}}_\varphi}
\newcommand{\Cpe}{\mathcal{C}_\varphi^e}
\newcommand{\Cpp}{\mathcal{C}_\varphi^\psi}
\newcommand{\Cppe}{\mathcal{C}_\varphi^{\psi e}}
\newcommand{\Cpte}{\Cpt^e}

\newcommand{\B}{\mathcal{B}}
\newcommand{\Bt}{\widetilde{\mathcal{B}}}
\newcommand{\Bte}{\Bt^e}
\newcommand{\Btp}{\Bt^\psi}
\newcommand{\Btpe}{\Bt^{\psi e}}
\newcommand{\Et}{\widetilde{\mathcal{E}}}
\newcommand{\Sc}{\mathcal{S}} 
\newcommand{\Sce}{\mathcal{S}^e} 
\newcommand{\Scp}{\mathcal{S}^\psi} 
\newcommand{\Scpe}{\mathcal{S}^{\psi e}} 
\newcommand{\Sct}{\widetilde{\mathcal{S}}} 
\newcommand{\Scte}{\widetilde{\mathcal{S}}^e} 
\newcommand{\Sctp}{\widetilde{\mathcal{S}}^\psi} 
\newcommand{\Sctpe}{\widetilde{\mathcal{S}}^{\psi e}} 

\author[Sandro Coriasco and Ren\'e Schulz]{Sandro Coriasco and Ren\'e Schulz}

\begin{document}

\title[$\SG$-Lagrangian submanifolds]{$\SG$-Lagrangian submanifolds\\and their parametrization}



\begin{abstract}
	We continue our study of tempered oscillatory integrals $I_\varphi(a)$, 
	here investigating the link with a suitable
	\textit{symplectic structure at infinity}, which we
	describe in detail. We prove adapted versions of the classical theorems,
	which show that tempered distributions
	of the type $I_\varphi(a)$ are indeed linked to suitable \textit{Lagrangians
	extending to infinity}, that is, extending up to the boundary and in particular the corners of a 
	compactification of $T^*\RRd$ to $\BBd\times\BBd$.
	In particular, we show that such Lagrangians can always be parametrized by 
	non-homogeneous, regular phase functions, globally defined on some $\RRd\times\RRs$.
	We also state how two such phase functions parametrizing the same Lagrangian
	may be considered equivalent \textit{up to infinity}.
\end{abstract}

\maketitle

\tableofcontents

\pagestyle{fancy}
\fancyhf{}
\fancyhead[RO,LE]{\thepage}
\fancyfoot[LO]{}
\fancyfoot[RE]{}
\fancyhead[RE]{\slshape\nouppercase{\rightmark}}
\fancyhead[LO]{\slshape\nouppercase{\leftmark}}
\renewcommand{\headrulewidth}{0.5pt}

\normalem


\section{Introduction}
\label{sec:intro}

In his groundbreaking paper of 1971, \cite{Hormander6}, H"ormander established a calculus of \textit{Fourier integral operators} (FIOs) in terms of their Schwartz kernels, given by Lagrangian distributions, see also \cite{Duistermaat, Hormander3, Hormander4}. The theory then proved to have many important applications in various branches of mathematics, and especially in the theory of partial differential equations.\\
A main feature of that theory is the possibility to pass from distributional expressions given by oscillatory integrals in local coordinates $x\in \RRd$, to invariantly defined geometric objects on manifolds.
In local coordinates on some manifold $X$, an oscillatory integral is of the form
$$I_\varphi(a)=\int e^{i\varphi(x,\theta)} a(x,\theta)\,d\theta,$$
with the \textit{phase function} $\varphi$ being smooth and homogeneous of degree $1$ in $\theta$, and satisfying certain ellipticity conditions. The \textit{amplitude} $a\in S^{m}_{1,0}(X\times\RRd)$ is instead a \textit{H\"ormander symbol} of order $m\in\RR$.\\
The connection mentioned above is established as follows. It is possible to associate with $\varphi$ its \textit{set of stationary points},
\begin{equation}
	\label{eq:lphi}
	\Lambda_\varphi=\big\{\big(x,\nabla_x\varphi(x,\theta)\big)\,|\, \nabla_\theta\varphi(x,\theta)=0\big\},
\end{equation}
which contains all information about the position of singularities of the corresponding class of oscillatory integrals. Namely, we have 
$$\bigcup_{a\in S^m_{1,0}} \WFcl(I_\varphi(a))=\Lambda_\varphi,$$ 
where $\WFcl$ denotes H\"ormander's classical \textit{wave front set}. For a non-degenerate phase function, 
$\Lambda_\varphi$ turns out to be a conic Lagrangian submanifold of $T^*X\setminus\{0\}$.\\
Conversely, for any conic Lagrangian submanifold $\Lambda$, we may find a local phase function $\varphi$ parametrizing it, that is, in a suitable neighbourhood of any $p\in\Lambda$, $\Lambda=\Lambda_\varphi$, with $\Lambda_\varphi$ given by \eqref{eq:lphi}. The symbol $a(x,\theta)$ may be recovered - up to terms of lower order - by means of an associated, invariantly
defined, \textit{principal symbol map}. Thus, one is able to pass from oscillatory integrals associated with a phase function to the invariant class of \textit{Lagrangian distributions}, associated with a corresponding Lagrangian submanifold of $T^*X\setminus\{0\}$. This opened up the theory of local FIOs to the possibility of being extended to manifolds.\\
Using this identification, many issues of the theory, such as criteria for the composition of two FIOs being again a FIO, can be expressed in terms of geometrical conditions on the involved Lagrangian submanifolds. This theory of Fourier integral operators is well suited to be applied either in small open neighbourhoods
of points, or on compact manifolds. In order to treat non-compact manifolds at the same level of efficiency, bounds on the involved distributional kernels, such as temperedness, need to be taken into account.\\
While the subject of $L^p$-continuity of FIOs on $\RRd$ has been studied in many global classes of FIOs, by imposing various bounds on the (derivatives) of the involved phase functions and symbols, see e.g. \cite{Asada}, \cite{AF}, \cite{CNR1}, \cite{CNR2} and \cite{RuSu}, the only approaches known to us that generalize the classical propagation of singularities are in the framework of the $\SG$-classes (or $^{sc}$-classes).\footnote{While \cite{CNR3} is concerned with the propagation of singularities under the action of Fourier Integral Operators, the wave front set under consideration is not a generalization of the classical one, but the independent notion of \textit{Gabor wave front set}.} This class of symbols was introduced by H.O.~Cordes \cite{Cordes} and C.~Parenti \cite{Parenti}, see also R.~Melrose's \textit{scattering calculus} in \cite{Melrose1,MZ}. In this framework it is possible to define a wave front set which turns out to be a generalization of the classical wave front set, in the sense that it also encodes singularities ``at infinity'', that is, those caused, for instance, by growth/decay and fast oscillations for $|x|\to+\infty$, see \cite{Cordes,CJT2,CM,Melrose1}. Propagation results in the scattering approach were given in \cite{HV,HW}, where operators with kernels that are \textit{Legendrian distributions}, see \cite{MZ}, are discussed. \\
In \cite{Coriasco1,Coriasco2}, see also \cite{Andrews,CR}, $\SG$-FIOs were introduced on $\RRd$ and their propagation of singularities is studied in \cite{CJT4,CM}.\\
The approach pursued in the present paper is a further generalization of the classical theory in terms of the $\SG$-calculus on $\RRd$, focusing on the properties of the involved phase functions and of the corresponding generalized Lagrangian submanifolds. The advantage is that our results can be formulated in terms similar to the classical ones, while still allowing a broad class of phase functions and including ``singularities at infinity''. An example of a distribution that may be treated from this point of view is the so-called two-point function arising in the study of the Klein-Gordon equation.\\
We note that the approach of \cite{HV,HW,Melrose2}, which is formulated in the language of $\textrm{sc}$-geometry on \textit{asymptotically flat}, or \textit{scattering} manifolds, while being related to the present analysis, is different from it. A major distinction is that our typical phase functions give rise to Lagrangian type singularities in all three components of the compactified cotangent bundle and the associated distributions are not smooth functions like the Legendrian distributions in \cite{MZ}. In fact, the above mentioned two-point function is not a smooth function, thus not a Legendrian distribution in the sense of \cite{MZ}, but admits Lagrangian type singularities in the interior as well as Legendrian type singularities at infinity.
\\
In \cite{CoSc} the authors have established a theory of tempered oscillatory integrals, which may be viewed as the local version of distributions arising from the geometric structures presented below. The involved objects extend the theory of classical oscillatory integrals, in the sense that they are tempered, and that their \textit{global} singularities may be understood in terms of the \textit{global set of stationary points} of their phase functions. The phase functions are assumed to be (inhomogeneous) $\SG$-symbols, whose derivatives satisfy an ellipticity condition.
Here the theory is complemented with the geometric picture, under the (natural) additional assumption that the phase function $\varphi$ is $\SG$-classical, that is a $\SG$-symbol of order $(1,1)$ which admits polyhomogeneous expansions. We note that even in this case the distributions under consideration differ from Legendrian distribution. In fact, by \cite[Proposition 10]{MZ}, the singularities of the Fourier transforms of Legendrian distributions on Euclidean spaces are contained in compact sets, a feature that is not true for our class of distributions. We discuss how the global set of stationary points of a non-degenerate $\SG$-classical phase functions form generalized Lagrangian submanifolds, which are submanifolds of a compactification of $T^*\RRd$, a manifold with corners, which turns out to be the natural environment within which to perform our analysis. In particular, we prove that the generalized Lagrangian submanifolds mentioned above can always be parametrized by $\SG$-classical phase functions and examine when two such parametrizations may be regarded as equivalent.\\
\\
We mention that some of the results of this paper have appeared in the thesis of the second author \cite{Schulz}. In subsequent works the authors will address the actual calculus of $\SG$-Lagrangian distributions and FIOs, with emphasis to the principal symbol maps and applications to differential equations. \\
\\
The paper is organized as follows.
In Section \ref{sec:prel} we recall various preliminary definitions and results. In particular, we list some of the basic element of the $\SG$-calculus in Subsection \ref{sec:sgbasics}, and give special emphasis to the subcalculus of $\SG$-classical symbols in Subsection \ref{subs:sgcl}. In particular, we review in detail how the latter may be expressed in terms of an embedding $\iota:\RRd\hookrightarrow\BBd$, compactifying $\RRd$ into the closed unit ball centered in the origin. In Subsection \ref{sec:toi} we recall the definition of tempered oscillatory integrals and the results concerning their singularities, studied in detail in \cite{CoSc}.\\
In Section \ref{sec:subm} we establish how the global set of (possible) singularities $\Lpt$ of a family of oscillatory integrals associated with a fixed $\SG$-phase function $\varphi$ may be regarded as a generalized Lagrangian submanifold. In Subsection \ref{subs:subm} we reformulate the results of Subsection \ref{sec:toi} in terms of 
$\iota$ and subsets of the ball $\BBd$. In Subsection \ref{sec:sympl} we associate these objects with the principal symbol of $\varphi$. Furthermore, we introduce a symplectic structure ``at infinity''. Finally, 
relying on the previous analysis, we show how $\Lpt$ may be regarded as a generalized Lagrangian in Subsection \ref{subs:sglagr}.\\
Our main theorems are proved in Section \ref{sec:sgpf}, where we show the converse of the
result proved in Subsection \ref{subs:sglagr}. Namely, given any $\SG$-Lagrangian $\tilde{\Lambda}$, it is always possible to find a $\SG$-classical phase function $\varphi$ locally parametrizing it, that is, 
$\tilde{\Lambda}=\Lpt$ in suitable neighbourhoods of points $p\in\Lambda$. Subsequently, we also prove a
theorem on the equivalence of phase functions in this context.\\
Finally, for the convenience of the reader, in the Appendix we give a summary of the differential calculus on manifolds with corners (with reference to \cite{MO}), which includes the results from that theory needed for our aims.
\subsection*{Acknowledgements}
The authors would like to express their gratitude for helpful advice and comments received  by Prof. D. Bahns, Dott. U. Battisti, Prof. B.-W. Schulze, Prof. A. Vasy and Prof. I. Witt.\\
The second author is grateful for financial support received by the German Research Foundation (Deutsche Forschungsgemeinschaft) through the Institutional Strategy of the University of Göttingen, in particular through the research
training group ``Graduiertenkolleg 1493'' and the Courant Research Center ``Higher Order Structures in Mathematics'', as well as for support by the ``Studienstiftung des deutschen Volkes''. 
The authors have been partially supported by the Gruppo Nazionale per l'Analisi Matematica, 
la Probabilit\`a e le loro Applicazioni (GNAMPA) of the Istituto Nazionale di Alta Matematica (INdAM),
grant ``Equazioni Differenziali a Derivate Parziali di Evoluzione e Stocastiche'' 
(Coordinator S. Coriasco, Dipartimento di Matematica ``G. Peano'', Universit\`a di Torino).


\section{Preliminary definitions and results}
\label{sec:prel}

\subsection{Basics of the $\SG$-calculus}
\label{sec:sgbasics}
In this section, we recall some preliminary definitions on the $\SG$-calculus. The subclass of classical symbols that admit polyhomogeneous expansions will be addressed in Section \ref{subs:sgcl}.\\

$SG$-pseudodifferential operators $A=a(x,D)=\Op(a)$ can be defined via the usual left-quantization
\[
	Au(x)= \frac{1}{(2 \pi)^{d}} \int e^{i x \cdot \xi} a(x, \xi) \hat u(\xi) d\xi,\quad u\in\Sw(\RR^d),
\]
starting from symbols $a(x,\xi) \in C^\infty(\RR^d\times\RR^d)$ with the property that, for arbitrary multiindices $\alpha,\beta\in\NNd$, there exist constants $C_{\alpha\beta}\ge0$ such that the estimates 
\begin{equation}
	\label{eq:disSG}
	|D_\xi^{\alpha}D_x^{\beta} a(x, \xi)| \leq C_{\alpha\beta} 
	\langle x\rangle^{m_e-|\beta|}\langle\xi\rangle^{m_\psi-|\alpha|}
\end{equation}
hold for fixed $m_e,m_\psi\in\RR$ and all $(x, \xi) \in \RR^d \times \RR^d$, where $\langle z \rangle=\sqrt{1+|z|^2}$, $z\in\RR^d$.
Symbols of this type belong to the class denoted by $\SG^{m_e,m_\psi}(\RR^d)$, which is a Fr\'echet space with a family of seminorms given by the ideal constants in \eqref{eq:disSG}, and the corresponding operators constitute the class
$L^{m_e,m_\psi}(\RR^d)=\Op\left(\SG^{m_e,m_\psi}(\RR^d)\right)$. In the sequel we will often simply write $\SG^{m_e,m_\psi}$ and $L^{m_e,m_\psi}$ when there can be no confusion about the spaces involved. \\
These classes of operators form a graded algebra, i.e. $L^{r_e,r_\psi}\circ L^{m_e,m_\psi}\subseteq L^{r_e+m_e,r_\psi+m_\psi}$, whose residual elements are operators with symbols in 
$$\displaystyle \SG^{-\infty,-\infty}(\RRd\times\RRd)= \bigcap_{(m_e,m_\psi) \in \RR^2} \SG^{m_e,m_\psi} (\RRd\times\RRd)=\Sw(\RR^{2d}),$$
that is, those having a Schwartz kernel in $\Sw(\RR^{2d})$, i.e. continuously mapping $\Sw^\prime(\RR^d)$ to $\Sw(\RR^d)$. An operator $A=\Op(a)\in L^{m_e,m_\psi}$ is called $SG$-elliptic if there exists $R\ge0$ such that $a(x,\xi)$ is invertible for $|x|+|\xi|\ge R$ and
\[
	a(x,\xi)^{-1}=O(\jap{x}^{-m_e}\jap{\xi}^{-m_\psi}).
\] 
Operators in $L^{m_e,m_\psi}$ act continuously from $\Sw(\RR^d)$ to itself, and extend as continuous operators from $\Sw^\prime(\RR^d)$ to itself and from $H^{s_e,s_\psi}(\RR^d)$ to $H^{s_e-m_e,s_\psi-m_\psi}(\RR^d)$, where $H^{t_e,t_\psi}(\RR^d)$, $t_e,t_\psi\in\RR$, denotes the weighted Sobolev space
\begin{align*}
  	H^{t_e,t_\psi}(\RR^d)&= \{u \in \Sw^\prime(\RR^{n}) \colon \|u\|_{t_e,t_\psi}= \|\Op(\pi_{t_e,t_\psi})u\|_{L^2}< \infty\},
 	\\
  	\pi_{t_e,t_\psi}(x,\xi)&=  \langle x\rangle^{t_e} \langle \xi \rangle^{t_\psi}.
\end{align*}
From their definition we have that $H^{s_e,s_\psi}(\RR^d)\hookrightarrow H^{r_e,r_\psi}(\RR^d)$ when $s_e\ge r_e$ and $s_\psi\ge r_\psi$, with compact embedding in case both inequalities are strict, while
\[
	\displaystyle\Sw(\RR^d)=\bigcap_{(s_e,s_\psi) \in \RR^2} H^{s_e,s_\psi}(\RR^d)
	\mbox{ and }
	\displaystyle\Sw^\prime(\RR^d)=\bigcup_{(s_e,s_\psi) \in \RR^2} H^{s_e,s_\psi}(\RR^d).
\]
An elliptic $SG$-operator $A \in L^{m_e,m_\psi}$ admits a parametrix $P\in L^{-m_\psi,-m_e}$ such that
\[
PA=I + K_1, \quad AP= I+ K_2,
\]
for suitable $K_1, K_2 \in L^{-\infty,-\infty}$, and it turns out to be a Fredholm operator.\\
We close this section by noting that $\SG$-operators may be introduced on more general spaces. In 1987, E.~Schrohe \cite{Schrohe1} introduced a class of non-compact manifolds, the so-called $SG$-manifolds, on which it is possible to transfer from $\RR^d$ the whole $SG$-calculus: in short, these are manifolds which admit a finite atlas whose changes of coordinates behave like symbols of order $(0,1)$ (see \cite{Schrohe1} for details and additional technical hypotheses). The manifolds with cylindrical ends are a special case of $SG$-manifolds, on which also the concept of $SG$-classical operator makes sense: moreover, the principal symbol of a $SG$-classical operator $A$ on a manifold with cylindrical ends $M$, in this case a triple $\sigma(A)=(\sigma_\psi(A),\sigma_e(A),\sigma_{\psi e}(A))$, has an invariant meaning on $M$, see Y.~Egorov and B.-W.~Schulze \cite{ES}, B.-W. Schulze \cite{Schulze}, R.~Melrose \cite{Melrose1,Melrose2} and Subsection \ref{subs:sgcl} below.

\subsection{Classical $\SG$-symbols}
\label{subs:sgcl}

We now introduce the subclass of the classical $\SG$ symbols
$\SG^{m_e,m_\psi}_\cl(\RR^d\times\RR^s)\subset \SG^{m_e,m_\psi}(\RR^d\times\RR^s)$.
Note that the only difference between the definition of the symbol space
$\SG^{m_e,m_\psi}(\RR^d\times\RR^s)$ and the ``standard'' $\SG$-symbols 
$\SG^{m_e,m_\psi}(\RR^d\times\RR^d)$, recalled in the Introduction, is that
we allow that the two independent variables $x,\, \xi$ belong to Euclidean spaces of (possibly) different dimensions $d,\,s,$ which naturally occurs in phase functions parametrizing Lagrangian submanifolds. In its classical formulation, the $\SG$-calculus was developed by Schulze, see \cite{Schulze}, to which we refer for most of the contents of this subsection. We begin by recalling the basic definitions and results (see also, e.g., \cite{ES,MSS06} for additional details and proofs). In the following, a $0$-excision function is a smooth function which identically vanishes in a neighbourhood of the origin, and which is
identically equal to $1$ outside a larger neighbourhood of the origin.

\begin{defn}
\label{def:sgclass-a}
\hspace{1pt}
\begin{itemize}
\item[i)]A symbol $a(x, \theta)$ belongs to the class $\SG^{m_e,m_\psi}_{\cl(x)}(\RR^d\times\RR^s)$ if there exist $a_{m_e-j, \bullet} (x, \theta)\in \mathscr{H}_x^{m_e-j}(\RR^d\times\RR^s)$, $j=0,1,\dots$, homogeneous functions of order $m_e-j$ with respect to the variable $x$, smooth with respect to the variable $\theta$, such that, for a $0$-excision function $\chi^e$,
\[
a(x, \theta) - \sum_{j=0}^{N-1}\chi^e(x) \, a_{m_e-j, \bullet} (x, \theta)\in \SG^{m_e-N, m_\psi}(\RR^d\times\RR^s), \quad N=1,2, \ldots;
\]
\item[ii)]A symbol $a(x, \theta) $ belongs to the class $\SG_{\cl(\theta)}^{m_e,m_\psi}(\RR^d\times\RR^s)$ if there exist $a_{\bullet, m_\psi-k}(x, \theta)\in \mathscr{H}_\theta^{m_\psi-k}(\RR^d\times\RR^s)$, $k=0,\,\dots$, homogeneous functions of order $m_\psi-k$ with respect to the variable $\theta$, smooth with respect to the variable $x$, such that, for a $0$-excision function $\chi^\psi$,
\[
a(x, \theta)- \sum_{k=0}^{N-1}\chi^\psi(\theta) \, a_{\bullet, m_\psi-k}(x,\theta) \in SG^{m_e, m_\psi-N}(\RR^d\times\RR^s), \quad N=1,2, \ldots
\]
\end{itemize}
\noindent
The symbols in $\SG^{m_e,m_\psi}_{\cl(x)}(\RRd\times\RRs)$ are called \textit{polyhomogeneous with respect to $x$ or $e$-polyhomogeneous}, those in $\SG^{m_e,m_\psi}_{\cl(\theta)}(\RRd\times\RRs)$ are called \textit{polyhomogeneous with respect to $\theta$ or $\psi$-polyhomogeneous}, respectively.
\end{defn}
\begin{defn}
\label{def:sgclass-b}
A symbol $a(x,\theta)$ is $\SG$-classical, and we write $a \in \SG_{\cl}^{m_e,m_\psi}(\RR^d\times\RR^s)=\SG_{\cl(x,\theta)}^{m_e,m_\psi}(\RR^d\times\RR^s)=\SG_{\cl}^{m_e,m_\psi}$, if
\begin{itemize}
\item[i)] there exist $a_{m_e-j,\bullet}(x, \theta)\in \mathscr{H}_x^{m_e-j}(\RR^d\times\RR^s)$ such that, 
for $0$-excision functions $\chi^e$, $\chi^e(x) \, a_{m_e-j, \bullet} (x, \theta)\in \SG_{\cl(\theta)}^{m_e-j, m_\psi}(\RR^d\times\RR^s)$ and
\[
a(x, \theta)- \sum_{j=0}^{N-1} \chi^e(x) \, a_{m_e-j, \bullet}(x, \theta) \in \SG^{m_e-N, m_\psi}(\RR^d\times\RR^s), \quad N=1,2,\dots;
\]
\item[ii)] there exist $a_{\bullet, m_\psi-k}(x, \theta)\in \mathscr{H}_\theta^{m_\psi-k}(\RR^d\times\RR^s)$ such that, 
for a $0$-excision function $\chi^\psi$, $\chi^\psi(\theta)\,a_{\bullet, m_\psi-k}(x, \theta)\in \SG_{\cl(x)}^{m_e, m_\psi-k}(\RR^d)$ and
\[
a(x, \theta) - \sum_{k=0}^{N-1} \chi^\psi(\theta) \, a_{\bullet, m_\psi-k} \in \SG^{m_e, m_\psi-N}(\RR^d), \quad N=1,2,\dots
\] 
\end{itemize}
We also set, when $s=d$, 
\[
	L_{\cl}^{m_e,m_\psi}=L_{\cl(x, \xi)}^{m_e,m_\psi}(\RR^d)=
	\Op(\SG^{m_e,m_\psi}_{\cl(x, \xi)}(\RR^d\times\RRd))=\Op(\SG^{m_e,m_\psi}_{\cl}).
\]
\end{defn}

\begin{rem}
	Definitions \ref{def:sgclass-a} and \ref{def:sgclass-b} can be extended, 
	in a natural way,
	from operators acting between scalars to operators acting between 
	(distributional sections of) vector bundles. In that case, matrix-valued symbols 
	are involved, whose entries satisfy the estimates \eqref{eq:disSG} and admit
	expansions in homogeneous terms as above, see \cite{Schulze}.
\end{rem}

\noindent
The next two results are very useful when dealing with $\SG$-classical symbols, see \cite{ES}.

\begin{thm}
	\label{thm:clexp}
	Let $a_{k} \in \SG_{\cl}^{m_e-k,m_\psi-k}(\RR^d\times\RR^s)$, $k =0,1,\dots$,
	be a sequence of $SG$-classical symbols
	and $a \sim \sum_{k=0}^\infty a_{k}$
	its asymptotic sum in the general $SG$-calculus.
	Then, $a \in \SG_{\cl}^{m_e,m_\psi}(\RR^d\times\RR^s)$.
\end{thm}
\begin{thm}
	\label{thm:repr}
	Let $\mathbb{B}^d= \{ y \in \RR^d: |y| \le 1 \}$  and let $\iota$ be a 
	diffeomorphism from $\RR^d$ to $\BBdo$ such that
	\[
	\iota(x) =  \frac{x}{|x|}\left(1-\frac{1}{|x|}\right) \quad \mbox{for} \quad |x| > 3,		
	\]
	whose inverse is given, for $1>|y|>2/3$,
	\[ 
	\iota^{-1}(y) = \displaystyle \frac{y}{|y|}(1-|y|)^{-1} .
	\]
	Choose also a smooth function $h\colon\BBd\to\RR_+$ such that
	$h(y)=|y|$ for $2/3<|y|\le1$ and 
	$1 - h(y) \not= 0$ for $|y|<2/3$, so that $\tilde{y}=1-h(y)$ is a boundary defining
	function on $\BBd$, i.e., it vanishes only at $\partial\BBd\simeq\SSSd$.
	
	Consider the map on $\SG^{m_e,m_\psi}_{\cl}(\RR^d\times\RR^s)$ given by 
	\begin{equation}
		\label{eq:sgclident}
		\begin{aligned}
		a(x,\theta)\mapsto b(y,\gamma) & = [1-h(y)]^{m_e} [1-h(\gamma)]^{m_\psi}
		a(\iota^{-1}(y), \iota^{-1}(\gamma))
		\\
		&=\tilde{y}^{m_e}\,\tilde{\gamma}^{m_\psi}
		\,[(\iota^{-1}\times\iota^{-1})^*a](y,\gamma)
		=\tilde{y}^{m_e}\,\tilde{\gamma}^{m_\psi}
		\,\tilde{a}(y,\eta).
		\end{aligned}
	\end{equation}
	Then, \eqref{eq:sgclident} extends to an isomorphism 
	\[
	\iota^{m_e,m_\psi}_\SG\colon \SG^{m_e,m_\psi}_{\cl}(\RR^d\times\RR^s)
	\to 
	\Sm(\BB^d \times \BB^s),
	\]
	that is, $\tilde{a}=(\iota^{-1}\times\iota^{-1})^*a
	\in\tilde{y}^{-m_e}\tilde{\gamma}^{-m_\psi}\Sm(\BB^d \times \BB^s)$.
\end{thm}
\begin{rem}
We remark that this isomorphism may be used to equip $\SGcl$ with a Fr\'echet topology.\\
To avoid confusion when different spaces are involved, we make systematic use of the following notation: 
\begin{itemize}
\item $y$ denotes ``variable-type'' elements of $\BBd$, $\eta$ denotes ``co-variable-type'' elements of $\BBd$, $\gamma$ denotes ``co-variable-type'' elements of $\BBs$,
\item  the corresponding elements of $\RRd\sqcup\RRdz$ are denoted by $x$ and $\xi$ and elements of $\RRs\sqcup(\RRs\setminus\{0\})$ are named $\theta$.
\end{itemize}
Subsets of $\BBd$ and $\BBs$ that correspond to subsets of $\RRd\sqcup(\RRdz)$ or $\RRs\sqcup(\RRs\setminus\{0\})$  are usually denoted by the same symbol equipped with a tilde.
\end{rem}
The following equivalent definition of $\SG$-classical symbol has been given by I. Witt in  
\cite{Witt}.
\begin{defn}
	\label{def:sgclw}
	Let $S^m_\cl(\RR^d)$, $m \in \RR{}$,
	denote the space of global classical symbols in one variable. This 
	means that $a \in S^m_\cl(\RR^d)$ if $a = a(x)$ is smooth on $\RR^d$, satisfies
	estimates like \eqref{eq:disSG} in the only variable $x$ and there exist 
	functions $a_{j} \in C^\infty(\RR^d \setminus\{0\})$, 
	$j \in \NN$, homogeneous of degree $m-j$,
	such that, for some $0$-excision function
    $\chi^e$, we have
	\[
	a(x) \sim \sum_{j=0}^\infty \chi^e(x) a_{j}(x).
	\]
	Then, for $m_e, m_\psi \in \RR$, $\SG_\cl^{m_e,m_\psi}(\RR^d\times\RR^s) = 
	S^{m_e}_{\cl}(\RR^d_x) \hat{\otimes}_{\pi} S^{m_\psi}_{\cl}(\RR^s_\theta)$, where
	$\hat{\otimes}_{\pi}$ denotes the completed tensor product.
\end{defn}
It easily turns out that $\SG$-classical symbols are closed under differentiation, sums and products.
Note also that the definition of $\SG$-classical symbol implies compatibility conditions for the terms of the expansions with respect to $x$ and $\xi$. In fact, defining the maps 
$\sigma_e^{m_e-j}$ and $\sigma_\psi^{m_\psi-k}$ on $\SG_{\cl(x)}^{m_e,m_\psi}$ and 
$\SG_{\cl(\theta)}^{m_e,m_\psi}$, respectively, in terms of the asymptotic expansions in Definition \ref{def:sgclass-a} as
\begin{align*}
	\sigma_e^{m_e-j}(a)(x, \theta) &= a_{m_e-j, \bullet}(x, \theta),\quad j=0, 1, \ldots, 
	\\
	\sigma_\psi^{m_\psi-k}(a)(x, \theta) &= a_{\bullet, m_\psi-k}(x, \theta),\quad k=0, 1, \ldots,
\end{align*}
it possible to prove that, for $a\in\SG^{m_e,m_\psi}_\cl$,
\[
\begin{split}
a_{m_e-j,m_\psi-k}=\sigma_{\psi e}^{m_e-j,m_\psi-k}(a)=\sigma_\psi^{m_\psi-k}(\sigma_e^{m_e-j}(a))= \sigma_e^{m_e-j}(\sigma_\psi^{m_\psi-k}(a)), \\
j=0,1, \ldots, \; k=0,1, \ldots
\end{split}
\]
Moreover, the algebra property of $\SG$-symbols and Theorem \ref{thm:clexp} imply that the composition of two $SG$-classical operators, i.e. operators with $\SG$-classical symbols denoted by $L^{m_e,m_\psi}_\cl(\RRd)$, is still classical. 

\begin{defn}
\label{def:sgphs}
For an operator $A=\Op(a)\in L^{m_e,m_\psi}_\cl(\RR^d)$, or for a symbol 
$a\in\SG_\cl^{m_e,m_\psi}(\RR^d\times\RR^s)$, the triples 
\[
\left.
\begin{array}{cl}
\sigma(A)&=(\sigma_e(A),\sigma_\psi(A),\sigma_{\psi e}(A))
\\
\sigma(a)&=(\sigma_e(a),\sigma_\psi(a),\sigma_{\psi e}(a))
\end{array}
\right\}:=(a_{m_e,\bullet}\,,\,
a_{\bullet,m_\psi}\,,\, a_{m_e,m_\psi})=:(a^e,a^\psi,a^{\psi e}),
\]
are called the \textit{principal symbol of $A$}, or, respectively, the 
\textit{principal symbol of $a$}. $a^\psi$ is called 
the homogeneous principal interior symbol and the pair $\{ a^e,a^{\psi e} \}$
the homogeneous principal exit symbol of $a$. 
\end{defn}
The principal symbol of an element of 
$\SG_\cl^{m_e,m_\psi}$ is well defined, and so also the principal symbol of
an element of $L^{m_e,m_\psi}_\cl$, in view of the following simple result.
\begin{prop}
	\label{prop:sgphswdef}
	Let $a\in\SG_{\cl}^{m_e,m_\psi}(\RR^d\times\RR^s)$, and let  
	$b_{m_e-j, \bullet}$, $b_{\bullet, m_\psi-k}$, $j,k=0,1,\dots$, 
	be arbitrary sequences of functions satisfying the requirements of
	Definition \ref{def:sgclass-b} with arbitrary excision functions
	$\omega^e$, $\omega^\psi$. Then,
	\[
		(a_{m_e,\bullet}\,,\,a_{\bullet,m_\psi}\,,\, a_{m_e,m_\psi})=
		(b_{m_e,\bullet}\,,\,b_{\bullet,m_\psi}\,,\, b_{m_e,m_\psi})=
		(a^e,a^\psi,a^{\psi e}).
	\] 
\end{prop}
The definition of principal symbol above keeps the usual multiplicative behaviour, that is, 
for any $A\in L^{r_e,r_\psi}_\cl$, $B\in L^{s_e,s_\psi}_\cl$, $r_e,r_\psi,s_e,s_\psi\in\RR$,
$\sigma(AB)=\sigma(A)\,\sigma(B)$, with component-wise product in the right-hand side. The same trivially holds for a product of two $\SG$-classical symbols, namely, for any
$a\in\SG^{r_e,r_\psi}_\cl$, $b\in\SG^{s_e,s_\psi}_\cl$, $\sigma(a\cdot b)=\sigma(a)\,\sigma(b)$.
It is also possible to canonically associate, with any $a\in\SG^{m_e,m_\psi}_\cl$, the \textit{principal part of $a$},
\begin{equation}
	\label{eq:SGprincipal}
	a_p(x,\theta)
	=\chi^e(x) a^e(x,\theta) +
	\chi^\psi(\theta)(a^\psi(x,\theta) - \chi^e(x) a^{\psi e}(x,\theta)),
\end{equation}
for $0$-excision functions $\chi^e,\chi^\psi$. One then finds 
$a-a_p\in\SG^{m_e-1,m_\psi-1}_\cl$. The next two propositions assert that
$a_p$ in \eqref{eq:SGprincipal} is indeed completely determined by $\sigma(a)$ and vice versa.  
\begin{prop}
	\label{prop:prspairs-a}
	Let $a\in\SG^{m_e,m_\psi}_\cl(\RRd\times\RRs)$. 
	Then, $\sigma^{m_e}_e(a)=\sigma^{m_\psi}_\psi(a)=0$ implies 
	$a\in\SG^{m_e-1,m_\psi-1}_\cl(\RRd\times\RRs)$.
\end{prop}
\begin{prop}
	\label{prop:prspairs-b}
	Let $(a^e,a^\psi)$ be a couple of functions satisfying the following assumptions:
	\begin{itemize}
	\item $a^e\in\mathscr{H}^{m_e}_x(\RRd\times\RRs)$ and, for a $0$-excision function $\chi^e$,
	$\chi^e(x)\,a^e(x,\theta)\in\SG^{m_e,m_\psi}_{\cl(\theta)}(\RRd\times\RRs)$;
	\item $a^\psi\in\mathscr{H}^{m_\psi}_\theta(\RRd\times\RRs)$ and, for a $0$-excision function 
	$\chi^\psi$, $\chi^\psi(\theta)\,a^\psi(x,\theta)\in\SG^{m_e,m_\psi}_{\cl(x)}(\RRd\times\RRs)$;
	\item $\sigma^{m_e}_e(\chi^\psi \, a^\psi)=\sigma^{m_\psi}_\psi(\chi^e \, a^e)=a^{\psi e}$.
	\end{itemize}
	Then, there exists $a\in\SG^{m_e,m_\psi}_\cl(\RRd\times\RRs)$
	such that $\sigma(a)=(a^e, a^\psi, a^{\psi e})$.
\end{prop}
Theorem \ref{thm:ellclass} below allows to express the ellipticity
of $\SG$-classical symbols and operators in terms of their principal symbol.
\begin{thm}
	 \label{thm:ellclass}
	An operator $A\in L^{m_e,m_\psi}_\cl(\RR^d)$ or a symbol 
	$a\in\SG^{m_e,m_\psi}_\cl(\RR^d\times\RR^s)$ is $\SG$-elliptic if and only if each
	element of the triple $\sigma(A)$, respectively $\sigma(a)$, is 
	non-vanishing on its domain of definition.
\end{thm}
In the following Definition \ref{def:wfs} we introduce some additional notation, which we will make systematical use of. 
\begin{defn}
	\label{def:wfs}
	We define the \textit{$\SG$-wave front space} as 
	$\Wt:=\partial(\BBd\times\BBd)=\Wte\sqcup \Wtp\sqcup\Wtpe$, where
	\begin{equation}
		\label{eq:defwt}
		\Wte:=\SSSd\times\left(\BBd\right)^o, \quad
		\Wtp:=\BBdo\times\SSS^{d-1}, \quad
		\Wtpe:=\SSSd\times\SSS^{d-1}.
	\end{equation}
	In a completely similar fashion, substituting $s$ in place of $d$ in the dimensions of the second
	factors in \eqref{eq:defwt}, we define $\Bt:=\partial(\BBd\times\BBs)=\Bte\sqcup \Btp\sqcup\Btpe$.
	We also set $\WSG=\WSG^e\sqcup\WSG^\psi\sqcup\WSG^{\psi e}$, with
	\begin{equation}
		\label{eq:defwsg}
		\begin{aligned}
		\WSG^e:=(\RRd\setminus\{0\})\times\RR^{d}, \;\, & \;\,
		\WSG^\psi:=\RR^d\times(\RRd\setminus\{0\}),
		\\
		\WSG^{\psi e}:=(\RRd\setminus\{0&\})\times(\RRd\setminus\{0\}),
		\end{aligned}
	\end{equation}
	and, again with $s$ in place of $d$ in the dimensions of the second factors of \eqref{eq:defwsg},
	$\B:=\B^e\sqcup\B^\psi\sqcup\B^{\psi e}$. Finally, we set $\Sc=\Sce\sqcup\Scp\sqcup\Scpe$, with
	\[
		\Sce=\SSSd\times\RRs,\quad\Scp=\RRd\times\SSS^{s-1},\quad\Scpe=\SSSd\times\SSS^{s-1}.
	\]
	and accordingly $\Sct$ as the union of
	\[
		\Scte=\SSSd\times(\BBs)^o,\quad\Sctp=(\BBd)^o\times\SSS^{s-1},\quad\Sctpe=\SSSd\times\SSS^{s-1}.
	\]	
	Moreover, with $\pi_{1,0}\in\SGcl^{1,0}(\RR^d\times\RR^s)$ and 
	$\pi_{0,1}\in\SGcl^{0,1}(\RR^d\times\RR^s)$ we denote the symbols
	\[
		\pi_{1,0}(x,\theta):=\jap{x}, \quad \pi_{0,1}(x,\theta):=\jap{\theta}.
	\]
	Finally, with any submanifold $M$ of $\SSSd$ or $\SSS^{s-1}$, we associate the conic manifold 
	$\Gamma(M)\subset(\RRdz)$ or $\RR^s\setminus\{0\}$ given by 
	\[
		\Gamma(M):=\RR_+\cdot M=\{\mu\cdot y\, \colon\, y\in M,\ \mu>0\}.
	\]
	Note that 
	\[
		\B^e=(\Gamma\times\mathrm{id})(\Sce), \quad 
		\B^\psi=(\mathrm{id}\times\Gamma)(\Scp), \quad
		\B^{\psi e}=(\Gamma\times\Gamma)(\Scpe).
	\]
	With $\Gamma$ we will also denote the map $y\mapsto\mu\cdot y$, for any vector $y\in\RRd$
	and a fixed $\mu>0$.
\end{defn}
In the sequel, we will systematically make use of the next two results. The first one shows that
derivatives with respect to variable and covariable commute with the \textit{principal
symbol map} $\sigma$ on $\SG^{m_e,m_\psi}_\cl$. The second one is a
characterization of the principal symbol of $a\in\SG^{m_e,m_\psi}_\cl$ in terms of the 
evaluation of the function
$\iota^{m_e,m_\psi}_\SG(a)\in\Sm(\BBd\times\BBs)$, defined in Theorem \ref{thm:repr},
at points in $\Bt$, then pull-back and extension by homogeneity.
By Theorem \ref{thm:ellclass}, $\SG$-ellipticity of $a$ can be then be expressed as the 
non-vanishing of $\iota^{m_e,m_\psi}_\SG(a)$ on $\Bt$.
\begin{prop}
	\label{prop:prsder}
	Let $a\in\SG^{m_e,m_\psi}_\cl(\RR^d\times\RR^s)$. Then, for any $\alpha\in\ZZ_+^d$
	and $\beta\in\ZZ_+^s$, 
	\[
		\sigma(\partial_x^\alpha\partial_\theta^\beta a(x,\theta))=
		\partial_x^\alpha\partial_\theta^\beta\sigma(a(x,\theta)).
	\]
\end{prop}
\begin{proof}
	We prove the result only for $\sigma_e(a)$, since the argument for $\sigma_\psi(a)$
	and $\sigma_{\psi e}(a)$ is completely similar. By Definition \ref{def:sgclass-b}, we
	have, for any $(x,\theta)\in\RR^d\times\RR^s$,
	\[
		a(x,\theta)=\chi^e(x)\,a^e(x,\theta)+p(x,\theta),
	\]
	with a $0$-excision function $\chi^e$ and a symbol $p\in\SG^{m_e-1,m_\psi}$.
	This implies, for any $\alpha\in\ZZ_+^d$ and $\beta\in\ZZ_+^s$,
	\begin{align*}
		(\partial^\alpha_x\partial^\beta_\theta a)(x,\theta)&=
		\chi^e(x)\,(\partial^\alpha_x\partial^\beta_\theta a^e)(x,\theta)
		\\
		&+\sum_{0<\kappa\le\alpha} c_{\alpha\kappa}\,(\partial^\kappa\chi^e)(x)\,
		(\partial^{\alpha-\kappa}_x\partial^\beta_\theta a^e)(x,\theta)
		+ (\partial^\alpha_x\partial^\beta_\theta p)(x,\theta)
		\\
		&=\chi^e(x)\,(\partial^\alpha_x\partial^\beta_\theta a^e)(x,\theta)
		+ q(x,\theta),
	\end{align*}
	with $q\in\SG^{m_e-|\alpha|-1,m_\psi-|\beta|}$. In fact, all the terms in the sum for
	$0<\kappa\le\alpha$ have compact support with respect to $x$, so that they all
	belong to $\SG^{-\infty,m_\psi-|\beta|}\subset\SG^{m_e-|\alpha|-1,m_\psi-|\beta|}$.
	Now note that, in view of Definition \ref{def:sgclass-b} and
	Proposition \ref{prop:sgphswdef}, 
	$\sigma_e(\partial^\alpha_x\partial^\beta_\theta a)(x,\theta)$ is the unique function
	$b^e(x,\theta)\in \mathscr{H}_x^{m_e-|\alpha|}$, such that 
	\[
		(\partial^\alpha_x\partial^\beta_\theta a)(x,\theta)=
		\chi^e(x) \, b^e(x,\theta)+q(x,\theta), \quad 
		q\in\SG^{m_e-|\alpha|-1,m_\psi-|\beta|},
	\]
	with $\chi^e(x) \, b^e(x,\theta)\in\SG^{m_e-|\alpha|,m_\psi-|\beta|}_{\cl(\theta)}$.
	Since $(\partial^\alpha_x\partial^\beta_\theta a^e)(x,\theta)$, by the
	hypotheses and the computations above, fulfills all such requirements, 
	we have the desired assertion.
\end{proof}
\begin{prop}
	\label{prop:prschar}
	With the notation of Theorem \ref{thm:repr} and Definition \ref{def:wfs}, for any 
	$a\in\SG^{m_e,m_\psi}_\cl(\RR^d\times\RR^s)$ we have
	\begin{align*}
		a^e|_{\Sce}&=
		(\mathrm{id}\times\iota)^*
		\left[\tilde{\gamma}^{-m_\psi}\cdot \iota^{m_e,m_\psi}_\SG(a)|_{\Bte}\right]
		=
		(\mathrm{id}\times\iota)^*
		\left[\left(\tilde{y}^{m_e}\cdot \tilde{a}\right)|_{\Bte}\right],
		\\
		a^\psi|_{\Scp}&=
		(\iota\times\mathrm{id})^*
		\left[\tilde{y}^{-m_e}\cdot \iota^{m_e,m_\psi}_\SG(a)|_{\Btp}\right]
		=
		(\iota\times\mathrm{id})^*
		\left[\left(\tilde{\gamma}^{m_\psi}\cdot \tilde{a}\right)|_{\Btp}\right],
		\\
		a^{\psi e}|_{\Scpe}&=
		(\mathrm{id}\times\mathrm{id})^*
		\left[\iota^{m_e,m_\psi}_\SG(a)|_{\Btpe}\right]
		=
		(\mathrm{id}\times\mathrm{id})^*
		\left[\left(\tilde{y}^{m_e}\tilde{\gamma}^{m_\psi}\cdot \tilde{a}\right)|_{\Btpe}\right],
	\end{align*}
	where $\mathrm{id}$ is the map identifying elements $y,\,\eta$ of $\partial\BBd$ and $\gamma\in\partial\BB^{s}$
	with the corresponding elements of the unit spheres $\SSSd\hookrightarrow\RR^d$ and 
	$\SSS^{s-1}\hookrightarrow\RR^s$, denoted by $x,\xi$ and $\theta$ respectively.
	Then, $a$ is $\SG$-elliptic if and only if $i^{m_e,m_\psi}_\SG(a)$ is nowhere vanishing 
	on $\B$.
\end{prop}
\begin{rem}
	It is clear that, by homogeneity, the values of $a^e$, $a^\psi$, $a^{\psi e}$ on
$\Sce$, $\Scp$ and $\Scpe$
	respectively, determine the corresponding components of $\sigma(a)$ on their whole domains of definition.
\end{rem}
\begin{proof}[Proof of Proposition \ref{prop:prschar}]
	By Definition \ref{def:sgclass-b} and Theorem \ref{thm:repr}, 
	we see that, for $|y|\ge 2/3$, $\gamma\in(\BB^s)^o$, $p\in\SG^{m_e-1,m_\psi}$,
	and a $0$-excision function $\chi^e$,
	\begin{align*}
		\iota^{m_e,m_\psi}_\SG(a)(y,\gamma)&=(1-|y|)^{m_e}\,\tilde{\gamma}^{m_\psi}\cdot
							\chi^e\left(\frac{y}{|y|}(1-|y|)^{-1}\right)\cdot
							a^e\left(\frac{y}{|y|}(1-|y|)^{-1},\iota^{-1}(\gamma)\right)
					\\
						 &\quad+(1-|y|)^{m_e}\,
						 \tilde{\gamma}^{m_\psi}\cdot
						 p\left(\frac{y}{|y|}(1-|y|)^{-1},\iota^{-1}(\gamma)\right)
						 \\
						 &=\tilde{\gamma}^{m_\psi}\cdot
							\chi^e\left(\frac{y}{|y|}(1-|y|)^{-1}\right)\cdot
							a^e\left(\frac{y}{|y|},\iota^{-1}(\gamma)\right)
					\\
						&\quad+(1-|y|)^{m_e}\,\tilde{\gamma}^{m_\psi}\cdot
						 p\left(\frac{y}{|y|}(1-|y|)^{-1},\iota^{-1}(\gamma)\right).							
	\end{align*}
	This implies immediately that, for $y\in\SSSd$, $\gamma\in\left(\BB^{s}\right)^o$,
	\begin{align*}
		\iota^{m_e,m_\psi}_\SG(a)&|_{\Bte}(y,\gamma)=
		\tilde{\gamma}^{m_\psi}\cdot a^e(y,\iota^{-1}(\gamma))
		\\
		&\Leftrightarrow
		a^e(y,\iota^{-1}(\gamma))=\left[\gamma^{-m_\psi}\cdot
		\iota^{m_e,m_\psi}_\SG(a)|_{\Bte}\right](y,\gamma)
		=\left(\tilde{y}^{m_e}\cdot\tilde{a}\right)|_{\Bte}(y,\gamma),
	\end{align*}
	which is equivalent to the first formula in the statement. The result for $a^\psi$ follows
	in the same way, exchanging the role of variable and covariable. 
	To prove the formula for $a^{\psi e}$, it is enough to notice that it also holds
	\begin{align*}
		a(x,\theta)&=\chi^e(x)\,a^e(x,\theta)+p(x,\theta)=
				\chi^e(x)\,[\chi^\psi(\theta)\,a^{\psi e}(x,\theta)+\tilde{q}(x,\theta)]+p(x,\theta)
				\\
				&=\chi^{e}(x)\,\chi^\psi(\theta)\, a^{\psi e}(x,\theta)+p(x,\theta)+q(x,\theta),
	\end{align*}
	with $0$-excision functions $\chi^e$, $\chi^\psi$, and symbols $p\in\SG^{m_e-1,m_\psi}$,
	$q\in\SG^{m_e,m_\psi-1}$. The desired result follows by restricting the related
	expression of $\iota^{m_e,m_\psi}_\SG(a)$ to $\B^{\psi e}$. Finally, the statement about 
	$\SG$-ellipticity of classical symbols is an immediate consequence of the formulae proved
	above, of Theorem \ref{thm:ellclass} and of the definition and properties of 
	$\tilde{y}$ and $\tilde{\gamma}$ from Theorem \ref{thm:repr}.
\end{proof}
We conclude the subsection by recalling the notion of \textit{local ellipticity} at points in $\Bt$
for $\SG$-classical symbols.
\begin{defn}
	\label{def:localell}
	A symbol $a\in\SG^{m_e,m_\psi}_\cl(\RRd\times\RRs)$ is 
	\textit{elliptic at $(y_0,\gamma_0)\in\Bt$} if $\iota^{m_e,m_\psi}_\SG(a)(y_0,\gamma_0)\not=0$.
\end{defn}
\begin{rem}
	\label{rem:localell}
	By Definition \ref{def:sgclass-b}, Theorem \ref{thm:ellclass} and Proposition \ref{prop:prschar},
	it follows that $a\in\SG^{m_e,m_\psi}_\cl(\RRd\times\RRs)$ is 
	elliptic at $(y_0,\gamma_0)\in\Bt$ if and only if we have
	$$|a(x,\theta)|\gtrsim\jap{x}^{m_e}\jap{\theta}^{m_\psi}\qquad\forall (x,\theta)\in V^\bullet$$
	where we define $V^\bullet$ depending on the component $\Bt^\bullet\subset\Bt$ for which we have $(y_0,\gamma_0)\in\Bt^\bullet$:\\
	For $R>0$ sufficiently large, we may set
	\begin{itemize}
		\item if $(y_0,\gamma_0)\in\Bte$, $V^e:=(\Gamma(V)\times K)\cap\{(x,\theta)\colon |x|\ge R>0\}$, for a 	suitable neighbourhood $V$ of $y_0$ in $\SSSd$ and a suitable bounded neighbourhood
		$K$ of $\iota^{-1}(\gamma_0)$ in $\RR^s$;
		\item if $(y_0,\gamma_0)\in\Btp$,
		$V^\psi:=( K\times\Gamma(V))\cap\{(x,\theta)\colon |\theta|\ge R>0\}$, for a suitable
		bounded neighbourhood $K$ of $\iota^{-1}(y_0)$ in $\RR^d$
		 and a suitable neighbourhood $V$ of $\gamma_0$ in $\SSS^{s-1}$;
		\item if $(y_0,\gamma_0)\in\Btpe$, $V^{\psi e}:=(\Gamma(V^1)\times \Gamma(V^2))\cap\{(x,\theta)\colon |x|,|\theta|\ge R>0\}$, 
		for a suitable neighbourhoods $V^1$ of $y_0$ in $\SSSd$ and 
		$V^2$ of $\gamma_0$ in $\SSS^{s-1}$.
	\end{itemize}
	More precisely, for a suitable symbol $\zeta\in\SG^{0,0}(\RR^d\times\RR^s)$, supported
	in a subset $V^\bullet$, of the type above, identically equal to
	$1$ in a smaller subset $U^\bullet\subset V^\bullet$ of the same type, 
	$\bullet\in\{e, \psi, \psi e\}$, it turns out that
	\textit{$a$ is $\SG$-elliptic with respect to $\zeta$}, cfr. \cite{Cordes,CJT2,CM}.
\end{rem}

\subsection{Tempered oscillatory integrals}
\label{sec:toi}
In this subsection we give a brief summary of the results we obtained in \cite{CoSc}. In that paper we have associated to a given (inhomogeneous) $\SG$-phase function $\varphi$ a family of tempered distributions, denoted by $I_\varphi(a)$, parametrized by amplitudes that are $\SG$-symbols and established a bound on their singularities. We begin by recalling the definition of such phase functions.
\begin{defn}
\label{def:phase}
Let $(n_e,n_\psi)\in\RR_+^2$. An element of  $\SG^{n_e,n_\psi}(\RRd\times\RRs)$ is called an (admissible) {\it $\SG$-phase function of order $(n_e,n_\psi)$} if it is real-valued and the associated function
\begin{equation}
	\label{eq:defeta}
	\begin{aligned}
	\pfe(x,\theta)&:=\langle x\rangle^2\,|\nabla_x \varphi(x,\theta)|^2+
	\langle \theta\rangle^2\,|\nabla_\theta \varphi(x,\theta)|^2
	\\
	&\phantom{:}=(|\pi_{1,0}\cdot\nabla_x \varphi|^2+|\pi_{0,1}\cdot\nabla_\theta \varphi|^2)(x,\theta)
	\end{aligned}
\end{equation}
is elliptic as an element of $\SG^{2n_e,2n_\psi}(\RRd\times\RRs)$, i.e. it satisfies, for some $R>0$,
\begin{equation}
	\label{eq:phaseineq}
	\pfe(x,\theta) \gtrsim \langle x\rangle^{2n_e} \langle \theta\rangle^{2n_\psi} 
	\text{ when } |x|+|\theta|\ge R.
\end{equation}
\end{defn}
\begin{rem}
Notice that we have not made any assumption on homogeneity and consequently these $\SG$-phase functions are in general inhomogeneous, as opposed to those that arise in the usual theory. Indeed, our approach is based on \cite{Zahn}, where a local theory of oscillatory integrals with inhomogeneous phase functions was developed.
\end{rem}
Using the notion of admissible $\SG$-phase function, we can now recall the definition of tempered oscillatory integrals given in \cite{CoSc}.
\begin{thm}
\label{thm:oscidef}
With any fixed admissible $\SG$-phase function $\varphi$ of order $(n_e,n_\psi)$ we may associate a map 
\[
	I_\varphi:\SG(\RR^d\times\RR^s)\rightarrow \Swd(\RR^d),
\] 
uniquely determined by the the following properties:
\begin{enumerate}
	\item $a\mapsto I_\varphi(a)$ is a linear map,
	\item If $a\in\Sw(\RR^d\times\RR^s)$, then $I_\varphi(a)$ coincides with the (absolutely convergent) integral
\begin{equation}
\label{eq:osciformdef}
I_\varphi(a)=\int_\RRs\, e^{i\varphi(x,\theta)}\,a(x,\theta)\,d\theta,
\end{equation}
\item the restriction of $I_\varphi$ to $\SGmm(\RR^d\times\RR^s)$ is a continuous map 
$$
	\SGmm(\RR^d\times\RR^s)\rightarrow\ \Swd(\RR^d).
$$
\end{enumerate}
We call the resulting distribution $I_\varphi(a)$ a $\SG$-{\it oscillatory integral}.
\end{thm}
For the above families of tempered oscillatory integrals we proved an inclusion for their so-called
\textit{$\SG$-wave front set}, which generalizes the corresponding one, valid in the standard setting, for 
H\"ormander's wave front set $\WFcl(u)$, namely\footnote{In the present paper we follow a notation close to the one used in \cite{Hormander3}, different from the one we adopted in \cite{CoSc}. 
In particular, in the original statement of Theorem \ref{thm:osciwf} proved there, $\Cpt$ was denoted by 
$\mathrm{M}_\varphi$, and $\Lpt$ by $\mathrm{SP}_\varphi$, respectively.} 
\[
	\pr_1(\WFcl(I_\varphi(a))=\mathrm{singsupp}(I_\varphi(a))\subset\pr_1(\Cp)
	\text{ and }
	\WFcl(I_\varphi(a))\subset\Lp,
\] 
see \cite{Hormander6}. In order to state our result in the $\SG$ setting, we first recall the definition of the $\SG$-wave front set. As before we make a strict distinction between the subsets of $\Wt$ and $\WSG$. Here we introduce the wave front set as a subset of $\Wt$, thus denoted $\WFSGt$:
\begin{defn}
Let $u\in\SwdRRd$. Then $\WFSGt(u)\subset\Wt$ is defined in terms of its complement as follows:
$$(y_0,\eta_0)\notin\WFSGt(u)\Leftrightarrow\,\exists\,A\in L^{0,0}_\cl\text{ elliptic at }(y_0,\eta_0)\text{ s.t. }Au\in\SwRRd.$$
\end{defn}
For more exposition and properties of this notion of wave front set, we refer to \cite{Cordes,CM,CoSc,Melrose1,Melrose2}. We now give the definition of the substitutes for the sets $\Cp$ and $\Lp$.
\begin{defn}
\label{def:Cpdef}
Let $\varphi\in\SG^{\nn}(\RRd\times\RRs)$ be an admissible $\SG$-phase function. Then 
$|\pi_{0,1}\cdot\nabla_\theta\varphi|^2\in\SG^{2n_e,2n_\psi}(\RRd\times\RRs)$ and 
$\Cpt$ denotes the set
\[
	\Cpt:=\left\{(y_0,\gamma_0)\in\Bt\,\colon\, |\pi_{0,1}\cdot\nabla_\theta\varphi|^2
	\text{ is not elliptic at }(y_0,\gamma_0)\right\}.
\]
Denote by $\pr_{\Cpt}$ the projection of $\Cpt\times\BBd\subset \BBd\times\BBs\times\BBd$ onto $\BBd\times\BBd$. We define the \textit{set of stationary phase points} of $\varphi$, $\Lpt\subset\Wt$, given in terms of its complement in $\Wt$, by
\begin{equation}
	\label{eq:spphi}
	\begin{aligned}
		(\Lpt)^c:=\{(y_0,\eta_0)\in\Wt\,\colon\, &\exists\, U\,\text{open neighbourhood of
		$(y_0,\eta_0)$ in $\BBd\times\BBd$}
	\\
		&\exists\, V\,\text{open neighbourhood of }\pr_{\Cpt}^{-1}(U)\text{ such that }
	\\
		& |\nabla_x\varphi(x,\theta)-\xi|\gtrsim \langle x\rangle^{n_e-1}\langle \theta\rangle^{n_\psi}+|\xi|
	\\
		& \text{for any }(x,\theta,\xi)\in (\iota^{-1}\times\iota^{-1}\times\iota^{-1})(V^o)\}.
	\end{aligned}
\end{equation}
Therein, $V^o$ denotes $V\cup (\BBdo\times(\BBs)^o\times\BBdo)$.
For convenience, we set
\[
	\Lpt^e=\Lpt\cap\Wt^e, \quad \Lpt^\psi=\Lpt\cap\Wt^\psi, \quad \Lpt^{\psi e}=\Lpt\cap\Wt^{\psi}.
\]
\end{defn}
Then, we have the following bounds for the singularities of the temperate oscillatory integral $I_\varphi(a)$ defined in Theorem \ref{thm:oscidef}.
\begin{thm}
\label{thm:osciwf}
Let $\varphi$ be an admissible $\SG$-phase function. Then, for any amplitude 
$a\in\SG^{m_e,m_\psi}(\RRd\times\RRs)$ we have the inclusions
\[
\pr_1(\WFSGt(I_\varphi(a)))\subset \pr_1(\Cpt)\text{ and }\WFSGt(I_\varphi(a))\subset \Lpt\,.
\]
\end{thm}
%


\medskip

\section{Submanifolds associated with $\SG$-classical phase functions}
\label{sec:subm}
\subsection{Submanifolds of $\BBd\times\BBs$}
\label{subs:subm}
In the Section \ref{sec:toi}, we have recalled the definition of oscillatory integrals for a very general class of phase functions and amplitudes. In the classical theory of H\"ormander, distributions locally defined by oscillatory integrals can be invariantly characterized as local representations of \textit{Lagrangian distributions} associated with the geometric object $\Lp$, which turns out to be a \textit{Lagrangian submanifold} of the cotangent bundle.\\
In the sequel, we will restrict our attention to $\SG$-classical phase functions of order $(1,1)$, namely, $\varphi\in\SG^{1,1}_{\cl}(\RRd\times\RRs)$, see \cite{CP1,CP2}, where it is possible to establish a similar geometric setup. In fact, under this stronger assumption, we will calculate the objects $\Cp$ and $\Lp$ in terms of the principal symbol of $\varphi$ and discuss their geometric properties. The approach in that will be to use Proposition \ref{prop:prschar} to associate with $\Cp$ and $\Lp$ conic manifolds, in the same way as one associates with a classical symbol its homogeneous principal symbols.
Recalling the existence of a canonical principal part for classical $\SG$-symbols, defined in \eqref{eq:SGprincipal}, we can write 
\[
\varphi(x,\theta)=\chi^e(x)\,\varphi^e(x,\theta)+\chi^\psi(\theta)\,\varphi^\psi(x,\theta)
-\chi^e(x)\,\chi^\psi(\theta)\,\varphi^{\psi e}(x,\theta)+r_\varphi(x,\theta)
\]
with 
\begin{itemize}
\item $\varphi^e\in\mathscr{H}_x^{1}(\RR^d\times\RR^s)$ being $\psi$-polyhomogeneous,
\item $\varphi^\psi\in\mathscr{H}_\theta^{1}(\RR^d\times\RR^s)$ being $e$-polyhomogeneous,
\item $\varphi^{\psi e}=\sigma_\psi^1(\chi^e\,\varphi^e)=\sigma_e^1(\chi^\psi\,\varphi^\psi)$,
\item $r_\varphi\in\SG^{0,0}_\cl(\RRd\times\RRs)$.
\end{itemize}
Since $e^{ir_\varphi}\in\SGzz$, we may absorb the $r_\varphi$ part of the phase function in an oscillatory integral into the amplitude. We are thus reduced to the case of studying phase functions of the form
\[
\varphi(x,\theta)=\chi^e(x)\,\varphi^e(x,\theta)+\chi^\psi(\theta)\,\varphi^\psi(x,\theta)-
\chi^e(x)\,\chi^\psi(\theta)\,\varphi^{\psi e}(x,\theta).
\]
Using Proposition \ref{prop:prschar}, we may obtain a representation of $\varphi(x,\theta)$ as a function on $(\BBd\times\BB^s)^o$ since $\tilde{\varphi}:=(\iota^{-1}\times\iota^{-1})^*\varphi\in\tilde{y}^{-1}\tilde{\gamma}^{-1}\Sm(\BBd\times\BB^s).$
\begin{rem}
\label{rem:ESiso}
This procedure, i.e. the use of Proposition \ref{prop:prschar}, allows us to work simply with smooth functions on the product of two balls instead of symbols. However, one has to be
careful when differentials are involved, since we have
\begin{align*}
\tilde{\gamma}\widetilde{\nabla_x\varphi}:=\left(\iota^{0,1}_\SG(\nabla_x \varphi)\right)(y,\gamma)&=
\tilde{\gamma}\,\nabla_y\tilde{\varphi}(y,\gamma)\cdot \Big(\frac{d\iota^{-1}(y)}{dy}\Big)^{-1}.\\
\tilde{y}\widetilde{\nabla_\theta\varphi}:=\left(\iota^{1,0}_\SG(\nabla_\theta \varphi)\right)(y,\gamma)&=
\tilde{y}\,\nabla_\gamma\tilde{\varphi}(y,\gamma)\cdot \Big(\frac{d\iota^{-1}(\gamma)}{d\gamma}\Big)^{-1}.
\end{align*}
Therein, the gradients are seen as vectors whose entries are $\SG$-symbols. We separate strictly between variables on $\RRd$ or $\RRs$ (denoted $x$, $\xi$, $\theta$) and such on $\BBd$ and $\BBs$ (denoted $y$, $\eta$, $\gamma$).
\end{rem}
\begin{lem}
\label{lem:Cptchar}
The condition that the associated function 
$$\Phi=|\pi_{1,0}\cdot\nabla_x \varphi|^2+|\pi_{0,1}\cdot\nabla_\theta \varphi|^2$$
is $\SG$-elliptic of order $(2,2)$ is equivalent to the condition that 
$\left(\tilde{\gamma}\widetilde{\nabla_x\varphi},\tilde{y}\widetilde{\nabla_\theta\varphi}\right)$ is nowhere vanishing on $\Bt$. Furthermore, we can write 
\begin{equation*}
\label{eq:Cptchar}
\Cpt=\{(y_0,\gamma_0)\in\Bt\,\colon\, \tilde{y}\widetilde{\nabla_\theta\varphi}(y_0,\gamma_0)=0\}.
\end{equation*}
\end{lem}
\begin{proof}
By Proposition \ref{prop:prschar}, we have that $\Phi$ is elliptic if and only if $\iota_\SG^{2,2}(\Phi)$ is nowhere vanishing on $\Bt$. We can also write,
\begin{align*}
\iota_\SG^{2,2}(\Phi)(y,\gamma)&=\tilde{\gamma}^2\tilde{y}^2 \left(\jap{x}^2|\nabla_x \varphi(x,\theta)|^2+\jap{\theta}^2|\nabla_\theta \varphi(x,\theta)|^2\right)\big|_{(x,\theta)=(\iota^{-1}(y),\iota^{-1}(\gamma))}\\
&=\left[|(\iota_\SG^{1,0}(\pi_{1,0})\cdot\iota_\SG^{0,1}(\nabla_x\varphi)|^2 + 
|(\iota_\SG^{0,1}(\pi_{0,1})\cdot\iota_\SG^{1,0}(\nabla_\theta\varphi)|^2\right](y,\gamma)
\end{align*}
Since $\pi_{1,0}$ and $\pi_{0,1}$ are elliptic, $(\iota_\SG^{1,0}(\pi_{1,0}))^2$ and 
$(\iota_\SG^{0,1}(\pi_{0,1}))^2$ are nowhere vanishing, which proves the first assertion. The characterization of $\Cpt$ follows by repeating the same argument for $|\pi_{0,1}\cdot\nabla_\theta\varphi|^2$, in view of Definition \ref{def:Cpdef}.
\end{proof}
We may then look at the map $\lambda_\varphi:\RRd\times\RRs\rightarrow\RRd\times\RRd$ given by $(x,\theta)\mapsto(x,\nabla_x\varphi(x,\theta))$. We want to find an analogue to this function on $(\BBd)^o\times(\BBs)^o$ that extends it to (parts of) the boundary. We start by considering the map
\[
	(y,\gamma)\mapsto \left(\iota^{-1}(y),\, \widetilde{\nabla_x\varphi}(y,\gamma)\right),\]
defined on $(\BBd)^o\times(\BBs)^o$. We may compactify the image space to $\BBd\times\BBd$, by means of the map $\iota\times\iota$, to look at the extension of
\begin{equation}
\label{eq:tildelambdadef}
\lpt\big|_{(\BBd)^o\times(\BBs)^o}=(\iota\times\iota)\circ \left((\iota^{-1}\times\iota^{-1})^*\lambda_\varphi\right)
\end{equation}
to the subset 
\begin{equation}
	\label{eq:lptext}
	\begin{aligned}
	\Et&=((\BBd)^o\times(\BBs)^o)\sqcup\Bt^e\sqcup\Bt_\mathrm{ell},
	\\
	\Bt_\mathrm{ell}&=\{(y_0,\gamma_0)\in\Btp\cup\Btpe\,\colon 
	\text{ $|\nabla_x\varphi|^2$ is elliptic at $(y_0,\gamma_0)$}
	\}.
	\end{aligned}
\end{equation}
\begin{rem}
\label{rem:tldef}
This construction may be visualized through the following commuting diagram:
\begin{center}
\begin{tikzpicture}
  \matrix (m) [matrix of math nodes,row sep=3em,column sep=4em,minimum width=2em]
  {
     \Et & \BBd\times\BBd \\
     \BBdo\times\left(\BBs\right)^o & \BBdo\times\BBdo \\
     \RRd\times\RRs & \RRd\times\RRs \\};
  \path[right hook->]
    (m-2-1) edge (m-1-1)
    (m-2-2) edge (m-1-2);
  \path[->]
    (m-1-1) edge node [above] {${\lpt}$} (m-1-2)
    (m-2-1) edge node [above] {${\lpt}$} (m-2-2)
    (m-3-1) edge node [above] {${\lambda_\varphi}$} (m-3-2)        
    (m-2-1) edge node [left] {$\iota^{-1}\times\iota^{-1}$}(m-3-1)
    (m-3-2) edge node [right] {$\iota\times\iota$}(m-2-2);
\end{tikzpicture}
\end{center}
\end{rem}
Indeed, we know by Theorem \ref{thm:repr} that the map 
$(\iota^{1,0}_\SG\times\iota^{0,1}_\SG)\lambda_\varphi:\BBd\times\BB^s\rightarrow \BBd\times\RRd$ given by 
\begin{equation}
\label{eq:tildelambdasmooth}
(y,\gamma)\mapsto \left(y,\tilde{\gamma}\,\widetilde{\nabla_x\varphi}(y,\gamma)\right)
\end{equation}
is smooth up to the boundary. We will show that, close to the boundary components of 
$\Et$, this property yields the desired extension of $\lpt$. 
\begin{prop}
	\label{prop:tl}
	$\lpt$ defined on $(\BBd)^o\times(\BBs)^o$ by \eqref{eq:tildelambdadef},
	can be extended as a smooth map to the subset $\Et\subset\BBd\times\BBs$ defined
	in \eqref{eq:lptext}.
\end{prop}
\begin{proof}
Since $\iota$ is a diffeomorphism, it is clear that $\lpt$ is smooth in the interior, i.e. on $(\BBd)^o\times\left(\BBs\right)^o$. So, it is enough that we look at \eqref{eq:tildelambdadef} for $|y|,|\gamma|>2/3$.
It is also clear that we have to prove the existence of the extension only for the second component 
of $\lpt$, since the first one coincides with $\pr_1$, the projection on the first set of variables, which is of course smoothly extendable from the
interior to the whole of $\BBd\times\BBs$.

By Definition \ref{def:sgclass-b} and Proposition \ref{prop:prsder}, we have, for a vector-valued symbol
$p\in\SG^{-1,1}$, 
\begin{equation}
\label{eq:lptbe}
\begin{aligned}
	\iota(\widetilde{\nabla_x\varphi}(y,\gamma))&=
	\iota\left(
		\nabla_x\varphi^e\left(\frac{y}{|y|}(1-|y|)^{-1},\frac{\gamma}{|\gamma|}(1-|\gamma|)^{-1}\right)
		+\tilde{p}(y,\gamma)
	\right)
	\\
	&=
	\iota\left(
		\nabla_x\varphi^e\left(\frac{y}{|y|},\frac{\gamma}{|\gamma|}(1-|\gamma|)^{-1}\right)
		+\tilde{p}(y,\gamma)\right).
\end{aligned}
\end{equation}
Then, $\lpt$ can be extended smoothly to 
\[
	A_1=\{y\in\BBd\colon 2/3<|y|\le1 \}\times\{\gamma\in\BBs\colon |\gamma|< r^\prime\},
\]
with arbitrary $r^\prime$, $1>r^\prime>2/3$. In fact, this is clearly true for the first term appearing in the argument of $\iota$ in the right hand side of \eqref{eq:lptbe}. For the second term, it is enough to observe that, by Theorem \ref{thm:repr}, for any $p\in\SG^{-1,1}$, $\tilde{p}\in\tilde{y}\tilde{\gamma}^{-1}\Sm(\BBd\times\BBs)$, that is, also $\tilde{p}$ is smooth on $A_1$. Moreover, the values of both such extensions to $A_1$ remain bounded, and $\iota$ is smooth on $\RRd$. This implies that $\lpt$ can be smoothly extended to any point in $\Bte$.

We now consider the subset of $\BBd\times\BBs$ given by 
\[
	A_2=\{y\in\BBd\colon 2/3<|y|<1 \}\times\{\gamma\in\BBs\colon |\gamma|> r\},
\]
$r^\prime>r>2/3$, so that, of course, $\Bt_\mathrm{ell}\subset A_2$. Observe that, again by Definition \ref{def:sgclass-b} and Proposition \ref{prop:prsder}, we can also write, for a vector-valued symbol $q\in\SG^{0,0}$,
\begin{equation}
\label{eq:lptbp}
\begin{aligned}
	\iota(\widetilde{\nabla_x\varphi}(y,\gamma))&=
	\iota\left(
		\nabla_x\varphi^\psi\left(\frac{y}{|y|}(1-|y|)^{-1},\frac{\gamma}{|\gamma|}(1-|\gamma|)^{-1}\right)
	+\tilde{q}(y,\gamma)
	\right)
	\\
	&=
	\iota\left(
		\nabla_x\varphi^\psi\left(\frac{y}{|y|}(1-|y|)^{-1},\frac{\gamma}{|\gamma|}\right)(1-|\gamma|)^{-1}
	+\tilde{q}(y,\gamma)
	\right).
\end{aligned}
\end{equation}
$\tilde{q}$ can be extended smoothly to $\BBd\times\BBs$, since, by Theorem
\ref{thm:repr}, for any $q\in\SG^{0,0}_\cl$, $i^{0,0}_\SG(q)=\tilde{q}\in\Sm(\BBd\times\BBs)$. 
By Propositions \ref{prop:prsder} and \ref{prop:prschar}, Definition \ref{def:localell} and Remark \ref{rem:localell}, at points $(y_0,\gamma_0)\in\Bt_\mathrm{ell}$ we have
\begin{align*}
	\text{either }& \quad 
	(y_0,\gamma_0)\in\Btp \text{ and }\nabla_x\varphi^\psi(\iota^{-1}(y_0),\gamma_0)\not=0,
	\\
	\text{or }& \quad 
	(y_0,\gamma_0)\in\Btpe \text{ and } \nabla_x\varphi^{\psi e}(y_0,\gamma_0)\not=0.
\end{align*}
In the former case, the norm of the first term in the argument of $\iota$ in the right hand side of
\eqref{eq:lptbp} tends to $+\infty$ when $|\gamma|\to1^-$.
Then, sufficiently close to $(y_0,\gamma_0)$ we have 
\begin{align} \label{eq:ltpexplicit}
\iota(\widetilde{\nabla_x\varphi})=
\frac{\widetilde{\nabla_x\varphi}}{|\widetilde{\nabla_x\varphi}|}\left(1-\frac{1}{|\widetilde{\nabla_x\varphi}|}\right)
= \frac{\tilde{\gamma}\widetilde{\nabla_x\varphi}}{|\tilde{\gamma}\widetilde{\nabla_x\varphi}|}\left(1-\frac{\tilde{\gamma}}{|\tilde{\gamma}\widetilde{\nabla_x\varphi}|}\right),
\end{align}
where $\tilde{\gamma}\widetilde{\nabla_x\varphi}=\iota^{0,1}_\SG(\nabla_x\varphi)$ is smooth up to the boundary. Moreover, 
\[
	\tilde{\gamma}\widetilde{\nabla_x\varphi}(y,\gamma)=\tilde{\gamma}(\iota^{-1}\times\iota^{-1})^* \nabla_x\varphi(y,\gamma)=
	\nabla_x\varphi^\psi\left(\frac{y}{|y|}(1-|y|)^{-1},\frac{\gamma}{|\gamma|}\right)
	+\tilde{\gamma}\cdot\tilde{q}(y,\gamma),
\]
so such an expression cannot vanish close to $(y_0,\gamma_0)$, since 
$|\nabla_x\varphi^\psi(\iota^{-1}(y_0),\gamma_0)|=k>0$ and 
$|\tilde{\gamma}\cdot\tilde{q}(y,\gamma)|<k/2$ for $(y,\gamma)\in V$,
suitably small neighborhood of $(y_0,\gamma_0)$, by
$|\tilde{\gamma}(\gamma_0)\cdot\tilde{q}(y_0,\gamma_0)|=0$. Then the smooth extendability of \eqref{eq:ltpexplicit} to points in $\Bt_\mathrm{ell}\cap\Btp$ follows.

The remaining case, that is, the result for $(y_0,\gamma_0)\in\Bt_\mathrm{ell}\cap\Btpe$,
follows in a similar way, writing
\begin{align*}
\iota(\widetilde{\nabla_x\varphi}(y,\gamma))&=
		\iota((\iota^{-1}\times\iota^{-1})^*\nabla_x\varphi(y,\gamma))
		\\
		&=
	\iota\left(
		\nabla_x\varphi^{\psi e}
		\left(\frac{y}{|y|}(1-|y|)^{-1},\frac{\gamma}{|\gamma|}(1-|\gamma|)^{-1}\right)
	+\tilde{p}(y,\gamma)+\tilde{q}(y,\gamma)
	\right)
	\\
	&=
	\iota\left(
		\nabla_x\varphi^{\psi e}\left(\frac{y}{|y|},\frac{\gamma}{|\gamma|}\right)(1-|\gamma|)^{-1}
	+\tilde{p}(y,\gamma)+\tilde{q}(y,\gamma)
	\right),
\end{align*}
with $p\in\SG^{-1,1}$, $q\in\SG^{0,0}$ and $\nabla_x\varphi(y_0,\gamma_0)\not=0$, so that
\[
	\tilde{\gamma}(\widetilde{\nabla_x\varphi}(y,\gamma)=
	\nabla_x\varphi^{\psi e}\left(\frac{y}{|y|},\frac{\gamma}{|\gamma|}\right)
	+\tilde{\gamma}\cdot\tilde{p}(y,\gamma)+\tilde{\gamma}\cdot\tilde{q}(y,\gamma),
\]
with the last two terms smoothly extendable to $(y_0,\gamma_0)$ and vanishing there. 

The proof is complete. 
\end{proof}
\begin{rem}
	\label{rem:tplcpt}
	Observe that, in view of the assumption \eqref{eq:phaseineq}, $\lpt$ is well defined in
	a neighborhood of $\Cpt$. In fact, by Propositions \ref{prop:prsder} and \ref{prop:prschar}, 
	Definition \ref{def:localell}, Remark \ref{rem:localell}, and 
	Lemma \ref{lem:Cptchar}, at points $(y_0,\gamma_0)\in\Cpt$ we necessarily have
	$\tilde{\gamma}\nabla_y\tilde{\varphi}(y_0,\gamma_0)\not=0\Leftrightarrow
	|\nabla_x\varphi|^2$ is elliptic at $(y_0,\gamma_0)$. Since this is equivalent to the fact
	that $\iota^{0,1}_\SG(\nabla_x\varphi)$ does not vanish at $(y_0,\gamma_0)$, the same holds,
	by continuity, in a neighborhood of $(y_0,\gamma_0)$ in $\Bt$.
\end{rem}
%
%
We are now able to obtain $\Lpt$ in terms of $\lpt$ and $\Cpt$:
\begin{lem}
\label{lem:Lptchar}
Let $\varphi\in\SG^{1,1}_\cl(\RRd\times\RRs)$ be a classical $\SG$-phase function. 
Then, we have $\Lpt=\lpt(\Cpt)$.
\end{lem}
For the sake of brevity, we omit the details of the proof, which follows the same methods used in
the proof of Proposition \ref{prop:tl} above.
We now impose a regularity condition on $\varphi$, namely, its \textit{$(\SG-)$non-degeneracy}.
\begin{defn}[Non-degenerate classical $\SG$-phase functions]
\label{def:nondegphase}
Let $\varphi\in\SGcl^{1,1}(\RRd\times\RRs)$ be a classical $\SG$-phase function. Then $\varphi$ is called \textit{non-degenerate} if the differentials $\left\{d\left(\tilde{y}\widetilde{\partial_{\theta_j}\varphi}|_{X}\right)\right\}_{j=1,\dots, s}$ form, for every $(y_0,\gamma_0)\in\,\Cpt$, a set of linearly independent vectors in $T^*_{(y_0,\gamma_0)}(X)$, where $X$ may be replaced by all boundary and corner components of $\BBd\times\BBs$, that is,
$$X\in\left\{\Bte,\, \Btp,\, \Btpe\right\}.$$
\end{defn}
Each of the boundary faces $\Bte$ and $\Btp$ are submanifolds (with boundary) of the manifold with corners $\BBd\times\BBs$ that intersect cleanly at their joint boundary $\Btpe$, that is, for every $(y_0,\gamma_0)\in\SSSd\times\SSS^{s-1}$ we have 
$$T_{(y_0,\gamma_0)}\Btpe=T_{(y_0,\gamma_0)}\Bte\cap T_{(y_0,\gamma_0)}\Btp.$$
We recall that, by Lemma \ref{lem:Cptchar}, $\Cpt$ is the set of boundary elements $(y_0,\gamma_0)$ jointly annihilated by $\tilde{y} \widetilde{\nabla_\theta\varphi}$, $j=1,\dots,s$. From that we are able to obtain a similar set-up for the different components of $\Cpt$, detailed in the next Proposition \ref{prop:ct}.
\begin{prop}
\label{prop:ct}
Let $\varphi\in\SGcl^{1,1}(\RRd\times\RRs)$ be a non-degenerate $\SG$-phase function. 
Then, the following properties hold true.
\begin{enumerate}
\item\label{it:Cpt1} The different components of $\Cpt$ are totally neat submanifolds of the corresponding boundary component in $\BBd\times\BBs$. That is, we have 
$$\Cpt=\underbrace{\Cpt^e}_{\subset\Bte} \cup\underbrace{\Cpt^\psi}_{\subset\Btp},$$ 
and their possible boundaries form a subset $\Cpt^{\psi e}$ of $\Btpe$. 
\item\label{it:Cpt2} The codimension of the respective component is always $s$, meaning $\dim(\Cpt^e)=\dim(\Cpt^\psi)=d-1$ and (if non-empty) $\dim(\Cpt^{\psi e})=d-2$.
\item\label{it:Cpt3} The tangent space to each face of $\Cpt^\bullet$ in $\Bt^\bullet$ may be calculated as 
$$\left\{v\in T_{(y_0,\gamma_0)}(\Bt^\bullet)\,\Big|\, \Big(d_{y,\gamma}\left(\tilde{y}\widetilde{\partial_{\theta_j}\varphi}\big|_{\Bt^\bullet}\right)\Big)v=0\ \forall j\in\{1,\dots,s\} \right\}.$$
\item\label{it:Cpt4} The intersection $\Cpt^\psi\cap\Cpt^e=\Cpt^{\psi e}$ is clean.
\end{enumerate}
\end{prop}
\begin{proof}
Statements \textit{(\ref{it:Cpt1})-(\ref{it:Cpt3})} are consequences of the regular value theorem for manifolds with corners, see Theorem \ref{thm:linindepdiff}. Then, also the cleanness of the intersection follows. 
\end{proof}
Since $\lpt$ is smooth up to the boundary in a neighborhood of $\Cpt$, we obtain a similar statement for 
$\Lpt$, in view of Lemma \ref{lem:Lptchar} and non-degeneracy, which causes $\lpt$ to be an immersion near $\Cpt$, in the sense of Theorem \ref{thm:imm}.
\begin{prop}
\label{prop:lpt}
Let $\varphi\in\SGcl^{1,1}(\RRd\times\RRs)$ be a non-degenerate $\SG$-phase function. 
Then, the following properties hold true.
\begin{enumerate}
\item\label{it:Lpt1} The different components of $\Lpt$ are each totally neat submanifolds of the corresponding boundary component $\BBd\times\BBs$. That is, we have 
$$\Lpt=\underbrace{\Lpt^e}_{\subset\Wte} \cup\underbrace{\Lpt^\psi}_{\subset\Wtp}$$ 
and their possible boundaries form a subset $\Lpt^{\psi e}$ of $\Wtpe$. 
\item\label{it:Lpt2} The codimension of the respective component is always $s$, meaning $\dim(\Lpt^e)=\dim(\Lpt^\psi)=d-1$ and (if non-empty) $\dim(\Lpt^{\psi e})=d-2$.
\item\label{it:Lpt3} The tangent space to each face of $\Lpt^\bullet$ in $\Wt$ may be calculated by means of the differential of $\lpt$, that is, via 
$$T\Lpt^\bullet=\left(d\left(\lpt|_{\Cpt^\bullet}\right)\right)T\Cpt^\bullet$$
\item\label{it:Lpt4} The intersection $\Lpt^\psi\cap\Lpt^e=\Lpt^{\psi e}$ is clean.
\end{enumerate}
\end{prop}
The aspect of clean intersection in Proposition \ref{prop:lpt} may be schematically visualized in $3$ dimensions, where the variables parallel to the corner, $(y^\parallel,\eta^\parallel)$, are compressed into one direction, see Figure \ref{fig:Lptintersect}.
\begin{figure}[ht!]
\begin{center}
\begin{tikzpicture}
  \node (A) at (2.5,3.5) {$\Wtp$};
  \node (B) at (6,0.5) {$\Wte$};
  \node (C) at (6.5,3) {$\Wtpe$}; 
  \draw[opacity=0.5] (3,1.5) -- (6,3);
  \draw[dashed,opacity=0.5] (1.5,2.25) -- (2.25,1.875);
  \draw[opacity=0.5] (2.25,1.875) -- (3,1.5) -- (3,0.75);
  \draw[dashed,opacity=0.5] (3,0.75) -- (3,0);
  \draw[dashed,opacity=0.5] (4.5,3.75) -- (5.25,3.375);
  \draw[opacity=0.5] (5.25,3.375) -- (6,3) -- (6,2.25);
  \draw[dashed,opacity=0.5] (6,2.25) -- (6,1.5);
  \draw[->,thick] (4,2) -- (5,2.5) node [below] {$\quad y^\parallel,\eta^\parallel$};
  \draw[->,thick] (4,2) -- (3,2.5) node [left] {$\tilde{y}$};
  \draw[->,thick] (4,2) -- (4,1) node [left] {$\tilde{\eta}$};
  \draw (4,2) .. controls (4,1) and (5,1) .. (4.5,0.5) node [right] {$\Lpt^e$};%
  \draw (4,2) .. controls (3,2.5) and (4,3) .. (3.5,3.5) node [right] {$\Lpt^\psi$};%
\end{tikzpicture}
\caption{Intersection of $\Lpt^e\subset\Wte$ and $\Lpt^\psi\subset\Wtp$ at the corner $\Wtpe$}
\label{fig:Lptintersect}
\end{center}
\end{figure}

\subsection{Conic submanifolds of $\WSG$ and symplectic structure at infinity}
\label{sec:sympl}
By means of the diffeomorphism $\iota^{-1}$, we may also write $\Cpt$ and $\Lpt$ in terms of more 
``classical'' objects, i.e. as subsets of $\B$ and $\WSG$, and characterize their components in terms 
of the principal parts of $\varphi$. Indeed, using also the map $\Gamma$ introduced in Definition \ref{def:wfs}, we have the following result.
\begin{lem}
\label{lem:Cpcl}
Let $\varphi\in\SGcl^{1,1}(\RRd\times\RRs)$ be a classical $\SG$-phase function. Then, we have
\begin{align*}
\Cpe&:=(\Gamma\times\iota^{-1})(\Cpt\cap\Bte) = \{(x_0,\theta_0)\in\WSG^e\,\colon\,\nabla_\theta\varphi^e(x_0,\theta_0)=0\}, \\
\Cpp&:=(\iota^{-1}\times\Gamma)(\Cpt\cap\Btp) = \{(x_0,\theta_0)\in\WSG^\psi\,\colon\,\nabla_\theta\varphi^\psi(x_0,\theta_0)=0\}, \\
\Cppe&:=(\Gamma\times\Gamma)(\Cpt\cap\Btpe) =  \{(x_0,\theta_0)\in\WSG^{\psi e}\,\colon\,\nabla_\theta\varphi^{\psi e}(x_0,\theta_0)=0\}.
\end{align*}
\end{lem}
\begin{proof}
By Definition \ref{def:Cpdef}, 
$$\Cpt=\big\{(y_0,\gamma_0)\in\Bt\,\colon\,|\nabla_\theta \varphi|^2\text{ is not elliptic at }(y_0,\gamma_0)\big\}.$$
By Theorem \ref{thm:ellclass} we have that $|\nabla_\theta \varphi|^2$ is elliptic at $(y_0,\gamma_0)\in\Bt$ if and only if the corresponding principal symbol is non-vanishing at the corresponding point $(x_0,\theta_0)\in\mathcal{B}$. By Proposition \ref{prop:prsder}, we have 
$$\sigma_\bullet(\nabla_\theta \varphi)=\sigma_\bullet\left(\sum_{j=1}^s |\partial_{\theta_j} \varphi|^2\right)=\sum_{j=1}^s |\partial_{\theta_j} \sigma_\bullet(\varphi)|^2,$$
for any of the components $\sigma_\bullet\in\{\sigma_e,\sigma_\psi,\sigma_{\psi e}\}$ of the principal symbol, and the assertion follows.
\end{proof}
\begin{rem}
Note that the $\Cpp$-component coincides with the standard notion $\Cp$ for a homogenous phase function $\varphi^\psi$.
\end{rem}
\begin{lem}
\label{lem:Lpchar}
Similarly to Lemma \ref{lem:Cpcl}, we define the triple $(\Lpe,\Lpp,\Lppe)\subset(\WSG^e,\WSG^\psi,\WSG^{\psi e})$ given by
\begin{align*}
\Lpe&:=\left\{\left((x,\nabla_x\varphi^e(x,\theta)\right)\,\big|\,\exists\,(x,\theta)\in\B^e:\ \nabla_\theta\varphi^e(x,\theta)=0\right\},\\
\Lpp&:=\left\{\left((x,\nabla_x\varphi^\psi(x,\theta)\right)\,\big|\,\exists\,(x,\theta)\in\B^\psi:\ \nabla_\theta\varphi^\psi(x,\theta)=0\right\},\\
\Lppe&:=\left\{\left((x,\nabla_x\varphi^{\psi e}(x,\theta)\right)\,\big|\,\exists\,(x,\theta)\in\B^{\psi e}:\ \nabla_\theta\varphi^{\psi e}(x,\theta)=0\right\}.
\end{align*}
Then we have $\Lpe=(\Gamma\times\iota^{-1})(\Lpt^e)$, $\Lpp=(\iota^{-1}\times\Gamma)(\Lpt^\psi)$ and $\Lppe=(\Gamma\times\Gamma)(\Lpt^{\psi e})$.
\end{lem}
\begin{proof}
We start with the proof for $\Lpp$, which coincides with the classical definition of the manifold of stationary points for a classical homogeneous phase function. We have, by Lemma \ref{lem:Lptchar},
$$(\iota^{-1}\times\Gamma)(\Lptp)=[(\iota^{-1}\times\Gamma)\circ\lpt](\Cpt^\psi).$$
By Lemma \ref{lem:Cptchar} we have $\tilde{y}\widetilde{\nabla_\theta\varphi}(y,\gamma)=0$ on $\Cpt$. Thus, in view of the same Lemma, $\tilde{\gamma}\widetilde{\nabla_x\varphi}(y,\gamma)\neq 0$. 
Recalling \eqref{eq:ltpexplicit} from the proof of Proposition \ref{prop:tl} and using the fact that $\tilde{\gamma}$ vanishes on $\Cpt^\psi$ and Remark \ref{rem:ESiso}, we can write
\begin{align*}
(\iota^{-1}\times\Gamma)(\Lptp)&=\left\{\left((\iota^{-1}(y),\mu \frac{\tilde{\gamma}\widetilde{\nabla_x\varphi}(y,\gamma)}{|\tilde{\gamma}\widetilde{\nabla_x\varphi}(y,\gamma)|}\right)\bigg|(y,\gamma)\in\Cpt^\psi,\mu>0\right\}\\
&=\left\{\left((x,\mu \frac{\nabla_x\varphi^\psi(x,\theta)}{|\nabla_x\varphi^\psi(x,\theta)|}\right)\bigg|(x,\theta)\in\Cpp,\mu>0\right\}
\end{align*}
where we have made use of the characterization of the principal symbol in Proposition \ref{prop:prschar} and of the commutativity property of differentiation and principal symbol map
in Proposition \ref{prop:prsder}. Making use of the homogeneity of $\varphi^\psi$, we may write this simply as
\begin{align*}
(\iota^{-1}\times\Gamma)(\Lptp)&=\left\{\left((x,\nabla_x\varphi^\psi(x,\theta)\right)\,\colon\,(x,\theta)\in\RRd\times(\RRs\setminus 0)\text{ and }\nabla_\theta\varphi^\psi(x,\theta)=0\right\},
\end{align*}
which is the definition of $\Lpp$, as claimed. In the same way we can write
\begin{align*}
(\Gamma\times\iota^{-1})(\Lpte)&=\big[\Gamma\times\iota^{-1})\lpt\big](\Cpte)\\
&=\left\{\big(\mu y,\widetilde{\nabla_x\varphi}(y,\gamma)\big)\, \colon\, (y,\gamma)\in\Cpte\right\}\\
&=\left\{\left(x,\nabla_x\varphi^e(x,\theta)\right)\, \colon\, (x,\theta)\in\Cpe\right\}.
\end{align*}
where we have again made use of Proposition \ref{prop:prschar}.\\
The characterization of the corner component $\Lppe$ follows in exactly the same way.
\end{proof}
As mentioned in the introduction to this subsection, it is a well-known result from the classical theory of Lagrangian submanifolds that $\Lpp$ as defined in Lemma \ref{lem:Lpchar} is a conic Lagrangian submanifold of $\RRd\times(\RRdz)$. We recall that a closed $d$-dimensional submanifold of $T^*\RRd\setminus\{0\}=\RRd\times(\RRdz)$ is called conic Lagrangian if it is conic in the second variable and the symplectic two-form $\omega$ vanishes over it. Equivalently one is able to obtain this by the vanishing of the canonical (or tautological) one-form $\alpha^\psi$ (we refer to Chapter 3.7. of \cite{Duistermaat} for the details). In what follows, we will obtain an analogous statement for $\Lambda^e$. 

From the discussion in \cite{Duistermaat} we deduce that the two formulations of a closed $d$-dimensional submanifold being conic Lagrangian are equivalent by noticing that (using local canonical coordinates)
\begin{align*}
d\alpha^\psi&=d(\xi dx)=d\xi\wedge dx=\omega\qquad \text{and}\\
i_{\varrho^\psi}\omega(\cdot)&=(d\xi\wedge dx)(\xi\cdot\partial_\xi,\cdot)=\xi dx=\alpha^\psi(\cdot),
\end{align*}
where the vector field $\varrho^\psi=\xi\cdot\partial_\xi$ can be invariantly obtained through the definition on $f\in \Sm(T^*M\setminus 0)$ by $\varrho^\psi(f)=\frac{d}{d\mu}f(\cdot,\mu \cdot)|_{\mu=1}$, that is, as the generator of the dilation in the fiber variable, see Section 21.1 in \cite{Hormander3}.
\begin{defn}
Let $M=\RRdz$. Define a vector field $\varrho^e$ on $T^*M$ by setting, for $f\in \Sm(T^*M)$,  $\varrho^e(f)=\frac{d}{d\mu}f(\mu\cdot,\cdot)|_{\mu=1}$. The \textit{exit-one-form} on $T^*M$ is defined as 
\[
	\alpha^e:=-i_{\varrho^e}\omega.
\]
\end{defn}
In local coordinates we have 
\[
	\varrho^e(f)=\frac{d}{d\mu}f(\mu x,\xi)|_{\mu=1}=x\cdot(\nabla_x f).
\]
Thus we have, in local canonical coordinates, $\alpha^e=-i_{\varrho^e}\omega=-x d\xi$, and therefore, again, $d\alpha^e=\omega$. We can now obtain
\begin{lem}
\label{lem:alphaelpe}
Let $\varphi$ be a non-degenerate classical $\SG$-phase function. Then $\alpha^e$ vanishes on $\Lpe$.
\end{lem}
\begin{rem}
We remark that, to our best knowledge, Lemma \ref{lem:alphaelpe} indeed requires its own proof, and cannot be simply ``deduced by symmetry'' from the classical theory, due to the ``asymmetrical definition'' of $\Lambda_\varphi$ with respect to $x$ and $\theta$.
\end{rem}
\begin{proof}
We adopt here the notation in \cite{Duistermaat}, and denote the induced coordinates on $T_{x}M$ by $\delta x$.

We first notice that $\Lpe$ is, by definition, the image of 
$$\Cpe=\{(x_0,\theta_0)\in\RRdz\times\RRs|\nabla_\xi \varphi^e(x_0,\theta_0)=0\}$$
which is a smooth manifold by non-degeneracy of $\varphi$, under the map $\lambda_\varphi^e=(\pr_1,\nabla_x\varphi^e)$.
We can thus calculate its tangent space in terms of that of the preimage\footnote{As in Lemma 2.3.2 of \cite{Duistermaat}, we can conclude from \eqref{eq:TCp} and \eqref{eq:TLp} that $(\pr_1,\nabla_x\varphi^e)$ is an immersion, and thus its image is an immersed $d$-dimensional conic submanifold.}.
$T_{(x,\theta)}\Cpe$ is given by
\begin{equation}
\label{eq:TCp}
(\delta x\cdot\nabla_x)\nabla_\theta\varphi+(\delta \theta\cdot\nabla_\theta)\nabla_\theta\varphi=0,
\end{equation}
and we thus have 
$$T_{(x,\nabla_x\varphi^e(x,\theta))}\Lambda^e_\varphi=
J(\pr_1,\nabla_x\varphi^e)\cdot T_{(x,\theta)}\Cpe.$$
Furthermore,
\begin{equation}
\label{eq:TLp}
J_{(x,\theta)}(\pr_1,\nabla_x\varphi^e)(\delta x,\delta \theta)=(\delta x, (\delta x\cdot \nabla_x)\nabla_x\varphi^e+(\delta\theta\cdot\nabla_\theta)\nabla_x\varphi^e).
\end{equation}
Computing $\alpha^e=x\cdot d\xi$ on such a vector, we see that
\begin{align}\notag
x\cdot (\delta x\cdot\nabla_x)&\nabla_x\varphi^e+x\cdot(\delta\theta\cdot\nabla_\theta)\nabla_x\varphi^e\\ \notag
&=\sum_{j,k} x_j (\delta x_k\partial_{x_k})\partial_{x_j}\varphi^e+\sum_j(\delta\theta\cdot\nabla_\theta)x_j\partial_{x_j}\varphi^e\\
&=\sum_{j,k} (\delta x_k\partial_{x_k})x_j\partial_{x_j}\varphi^e-\sum_{k}\delta x_k \partial_{x_k}\varphi^e+\sum_j(\delta\theta\cdot\nabla_\theta)x_j\partial_{x_j}\varphi^e. \label{eq:alphaeeval}
\end{align}
Since $\varphi^e$ is $1$-homogeneous in the first set of variables, by Euler's theorem for homogeneous functions this equals to
\begin{align*}
\eqref{eq:alphaeeval}&=\sum_{k} (\delta x_k\partial_{x_k})\varphi^e-\sum_{k}\delta x_k \partial_{x_k}\varphi^e+(\delta\theta\cdot\nabla_\theta)\varphi^e,\\
&=\delta\theta\cdot(\nabla_\theta\varphi^e)\stackrel{(x,\theta)\in\Cpe}{=}0
\end{align*}
This proves the assertion.
\end{proof}
Lastly, we observe that some additional properties that these kind of submanifolds, arising from $\SG$-classical phase functions, have, which limit their behaviour at infinity.
\begin{lem}
Let $\varphi\in\SGcl^{1,1}(\RRd\times\RRs)$ be a non-degenerate classical $\SG$-phase function. Then
\begin{enumerate}
\item The pairing $\langle x,\xi\rangle$ vanishes on $\Lambda^{\psi e}$.
\item $\Lambda^e_\varphi$ does not intersect $(\RRdz)\times\{0\}$.
\end{enumerate}
\end{lem}
\begin{proof}
On $\Lambda^{\psi e}$ we have, by Euler's theorem for homogeneous functions applied twice,
\[
	\langle x,\xi\rangle=\langle x,\nabla_x\varphi^{\psi e}(x,\theta)\rangle=\varphi^{\psi e}(x,\theta)
	=\theta\cdot\nabla_\theta\varphi^{\psi e}(x,\theta)=0.
\]
The second assertion follows from the characterization of $\Lambda^e_\varphi$ in Lemma \ref{lem:Lpchar}, since \eqref{eq:phaseineq} implies that if $\nabla_\theta\varphi^e(x,\theta)=0$ we have $\nabla_x\varphi^e(x,\theta)\neq 0$.
\end{proof}
\subsection{$\SG$-Lagrangian submanifolds}
\label{subs:sglagr}
A $\SG$-Lagrangian manifold $\tilde{\Lambda}$ is a ``submanifold'' of $\Wt$ with Lagrangian properties. Note that this is a slight abuse of notation since technically, $\Wt$ is not a manifold, but a pair of submanifolds of $\partial(\BBd\times\BBd)$ that intersect cleanly at their joint boundary, the corner of $\BBd\times\BBd$, see Figure \ref{fig:Lptintersect}. Therefore, we have to consider a number of cases when we define such a ``submanifold''. We consider the case where $\tilde{\Lambda}$ intersects the corner.
\begin{defn}
A \textit{Lagrangian submanifold} $\Lt$ of $\Wt$ is a pair of closed embedded submanifolds (with boundary) of $\BBd\times\BBd$, $\Lt=(\Lte,\Ltp)$, such that
\begin{itemize}
\item $(\Lte)^o\subset\Wte\setminus(\SSSd\times\{0\})$, $(\Ltp)^o\subset\Wtp$,
\item $\dim(\Lte)=\dim(\Ltp)=d-1$,
\item $(\Lte\cap\Ltp)=\partial\Lte=\partial\Ltp=:\Ltpe\subset\Wtpe$ (with $\dim(\Ltpe)=d-2$) and the intersection being clean,
\item on the associated conifications 
\[
\Lme:=(\Gamma\times\iota^{-1})\left((\Lte)^o\right),\;
\Lmp:=(\iota^{-1}\times\Gamma)\left((\Ltp)^o\right),\;
\Lmpe:=(\Gamma\times\Gamma)(\Ltpe),
\] 
we have
$$\alpha^e|_\Lme=0,\qquad\alpha^\psi|_\Lmp=0,\qquad\alpha^e|_\Lmpe=\alpha^\psi|_\Lmpe=0,$$
\item in canonical coordinates we have the \textit{conormality condition} $\langle x,\xi\rangle|_\Lmpe=0$.
\end{itemize}
The triple $(\Lme,\Lmp,\Lmpe)$ is then called a \textit{conic $\SG$-Lagrangian submanifold of $T^*\RRd$}.
\end{defn}
The ``degenerate cases'' are then straightforward to define. If there is no intersection in the corner, then one of the submanifolds (which will no longer have a boundary) may be empty, or they form two disjoint submanifolds of $\Wte$ and $\Wtp$ respectively.
We may sum up the results of the previous subsection as follows:
\begin{thm}
Let $\varphi\in\SGcl^{1,1}(\RRd\times\RRs)$ be a non-degenerate $\SG$-phase function. Then $\Lpt=(\Lpte,\Lptp)$ is a $\SG$-Lagrangian submanifold of $\Wt$.
\end{thm}
%

%
\section{Parametrization of $\SG$-Lagrangians}
\label{sec:sgpf}
In this section we prove our main results. We first investigate, in the next theorem, how one is always able 
to find a $\SG$-classical phase function to locally parametrize the $\SG$-Lagrangians.

\begin{thm}
Let $\Lt=(\Lte,\Ltp)$ be a $\SG$-Lagrangian submanifold. Then $\Lt$ is locally parametrizable by a non-degenerate $\SG$-classical phase function, that is, $\forall (y_0,\eta_0)\in\Lt$ there exist
\begin{enumerate}
\item a neighbourhood $\tilde{U}$ of $(y_0,\eta_0)$ in $\BBd\times\BBd$,
\item an open set $\tilde{V}\subset\BBd\times\BBs$,
\item a function $\tilde{\varphi}\in\tilde{\gamma}^{-1}\tilde{y}^{-1}\Sm(\tilde{U})$ such that the corresponding (locally defined) phase function $\varphi=(\iota\times\iota)^{*}(\tilde{\varphi})$ is non-degenerate,
\end{enumerate}
such that 
$$\Lt\cap\tilde{U}=\lpt\left(\big\{(y_0,\gamma_0)\in\tilde{V}\cap\Bt\,\colon\, (y_0,\gamma_0)\in\Cpt\big\}\right).$$
\end{thm}
\begin{proof}
We will only consider the case where $(y_0,\eta_0)\in\Ltpe$, since the other possible situations will be covered by the same argument. The outline of the proof is classical, cf. \cite{HW} and \cite{Hormander3}, but here some tools from the theory of manifolds with corners are
essential to achieve the result, as well as the extension of $\lpt$ and the symplectic structure ``at infinity'' discussed in Subsection \ref{sec:sympl}.\\
Let $(y_0,\eta_0)\in\Ltpe$. $\Ltpe$ is a $(d-2)$-dimensional embedded submanifold of 
$\SSSd\times\SSSd$ and we may assume, possibly after a rearrangement of variables in a neighbourhood $\tilde{U}$ of $(y_0,\eta_0)$, that $\Ltpe$ is parametrized as
$$\tilde{U}\cap\Ltpe=\left\{y^\prime,y^\pp,\sqrt{1-(y^\prime)^2-(y^\pp)^2},\sqrt{1-(\eta^\prime)^2-(\eta^\pp)^2},\eta^\prime,\eta^\pp\right\},$$
where, for some $s\leq d-1$, we have that $\eta^\prime=(\eta_2,\dots,\eta_{s})$ and $y^{\pp}=(y_{s+1},\dots,y_{d-1})$ are independent variables and the remaining variables,
\begin{align*}
y^\prime&=\tilde{Y}^{\psi e}(y^\pp,\eta^\prime),\\
\eta^\pp&=\tilde{H}^{\psi e}(y^\pp,\eta^\prime),
\end{align*}
are smoothly dependent on $(y^\pp,\eta^\prime)$. We may further assume that $y_d$ and $\eta_1$ do not vanish in the chosen coordinate neighbourhood, that is we have, for some $1\geq c>0$, ${y_d}>c$ and $\eta_1>c$.\\
Due to the clean intersection at the corner $\Ltpe=\Lte\cap\Ltp=\partial\Lte=\partial \Ltp$, that is $T_{\Ltpe}\Lte\cap T_{\Ltpe}\Ltp= T\Ltpe$, we may find, accordingly, parametrizations of $\Lte$ and $\Ltp$ near the corner point $(y_0,\eta_0)$, namely
\begin{align*}
\tilde{U}\cap\Lte=\left\{y^\prime,y^\pp,\sqrt{1-(y^\prime)^2-(y^\pp)^2},\eta_1,\eta^\prime,\eta^\pp\right\},\\
\tilde{U}\cap\Ltp=\left\{y^\prime,y^\pp,y_d,\sqrt{1-(\eta^\prime)^2-(\eta^\pp)^2},\eta^\prime,\eta^\pp\right\}.
\end{align*}
Here we have the independent coordinates $(y^\pp,\eta_1,\eta^\prime)$ on $\Lte$ and $(y^\pp,y_d,\eta^\prime)$ on $\Ltp$. The remaining variables on $\tilde{U}\cap\Ltp$ may be written as functions smooth up to the boundary,
\[
y^\prime=\tilde{Y}^e(y^\pp,\eta_1,\eta^\prime),\quad
\eta^\pp=\tilde{H}^e(y^\pp,\eta_1,\eta^\prime),
\]
and on $\tilde{U}\cap\Ltp$ as
\[
y^\prime=\tilde{Y}^\psi(y^\pp,y_d,\eta^\prime),\quad
\eta^\pp=\tilde{H}^\psi(y^\pp,y_d,\eta^\prime).
\]
By $\Lte\cap\Ltp=\partial\Lte=\partial\Ltp=\Ltpe$ we conclude that if 
$$\big(\eta_1,\eta^\prime,\tilde{H}^e(y^\pp,\eta_1,\eta^\prime)\big)\in\SSSd\text{ and }\big(\tilde{Y}^\psi(y^\pp,y_d,\eta^\prime),y^\pp,y_d\big)\in\SSSd$$
we have
\begin{align} \label{eq:yeypsi}
\tilde{Y}^e(y^\pp,\eta_1,\eta^\prime)&=\tilde{Y}^\psi(y^\pp,y_d,\eta^\prime)=\tilde{Y}^{\psi e}(y^\pp,\eta^\prime),\\
\label{eq:yeypsi2}
\tilde{H}^e(y^\pp,\eta_1,\eta^\prime)&=\tilde{H}^\psi(y^\pp,y_d,\eta^\prime)=\tilde{H}^{\psi e}(y^\pp,\eta^\prime).
\end{align}
This choice of coordinates induces coordinates on the associated conifications $\Lme=(\Gamma\times\iota^{-1})(\Lte)$ and $\Lmp=(\iota^{-1}\times\Gamma)(\Ltp)$. That is, we may take, as independent variables on $\Lme$,
\begin{align*}
x^\pp=(\mu y^\pp,\mu \sqrt{1-(y^\prime)^2-(y^\pp)^2}),\qquad
\xi^\prime=\iota^{-1}(\eta_1,\eta^\prime).
\end{align*}
In particular, $x^\pp$ may be defined implicitly in terms of the map
$$(y^\pp,\mu)\mapsto \left(\mu (\id\times\iota)^*\tilde{Y}^e(y^\pp,\xi^\prime),\mu y^\pp,\mu\sqrt{1-((\id\times\iota)^*\tilde{Y}^e(y^\pp,\xi^\prime))^2-(y^\pp)^2}\right).$$
We obtain that $x^\prime=\mu (\id\times\iota)^*\tilde{Y}^e(y^\pp,\xi^\prime)=:X^e(x^\pp,\xi^\prime)$ is a smooth function of $x^\pp$ and $\xi^\prime$ and polyhomogeneous in $\xi^\prime$, of maximal degree $0$. By $|(x^\prime,x^\pp)|=\mu$ it is further $1$-homogeneous in $x^\pp$. Similarly we have that $\xi^\pp=\iota^{-1}\left((\id\times\iota)^*\tilde{H}^e(y^\pp,\xi^\prime)\right)=:\Xi^e(x^\pp,\xi^\prime)$ is $0$-homogeneous in $x^\pp$ and polyhomogeneous in $\xi^\prime$.\\
We can thus write, locally near $(x_0,\xi_0)=(\id\times\iota^{-1})(y_0,\eta_0)$,
\[
\Lme=\left\{\big(X^e(x^\pp,\xi^\prime),x^\pp;\xi^\prime,\Xi^e(x^\pp,\xi^\prime\big)\right\}.
\]
In the same way we may write, in coordinates 
\begin{align*}
x^\pp=\iota^{-1}(y^\pp,y_d),\qquad
\xi^\prime=(\mu \eta_1,\mu \eta^\prime),
\end{align*}
that
\[
	\Lmp=\left\{\big(X^\psi(x^\pp,\xi^\prime),x^\pp;\xi^\prime,\Xi^\psi(x^\pp,\xi^\prime\big)\right\}.
\]
We now define phase functions parametrizing these conic submanifolds in the given neighbourhoods. We set
\begin{align}\label{eq:phiedef}
\phi^e(x,\xi)&=\langle x^\prime,\xi^\prime\rangle + \langle x^\pp,\Xi^e(x^\pp,\xi^\prime)\rangle,
\\ \label{eq:phipsidef}
\phi^\psi(x,\xi)&=\langle x^\prime,\xi^\prime\rangle -\langle X^\psi(x^\pp,\xi^\prime),\xi^\prime\rangle.
\end{align}
By the above definitions of $\Xi^e$ and $X^\psi$ we observe that $\phi^e$ is $1$-homogeneous in $x$ and 1-polyhomogeneous in $\xi$, whereas $\phi^\psi$ is $1$-homogeneous in $\xi$ and polyhomogeneous in $x$. In fact these functions, restricted to suitable neighbourhoods in $\SSSd\times\RRd$ and $\RRd\times\SSSd$, respectively, may be written as 
\begin{align}\label{eq:pte}
\phi^e(x,\xi)|_{\SSSd\times\RRd}&=(\id\times\iota)^{*}\underbrace{\left(\left\langle (y^\prime,y^\pp,y_d)\, ,\,\iota^{-1}\left(\eta_1,\eta^\prime,\tilde{H}^e(y^\pp,\eta_1,\eta^\prime))\right)\right\rangle\right)}_{=:\tilde{y}\cdot\tilde{\phi}^e|_{\Wte}}\\
\phi^\psi(x,\xi)|_{\RRd\times\SSSd}&=(\iota\times\id)^{*}\underbrace{\left(\left\langle \iota^{-1}(y^\prime)-\iota^{-1}\left(\tilde{Y}^\psi(y^\pp,y_d,\eta^\prime)\right),(\eta_1,\eta^\prime)\right\rangle\right)}_{=:\tilde{\eta}\cdot\tilde{\phi}^\psi|_{\Wte}}.\label{eq:ptp}
\end{align}
Using $\iota^{-1}(y)=\frac{y}{|y|}(1-|y|)^{-1}=\tilde{y}^{-1}\frac{y}{|y|}$ for large arguments and Proposition \ref{prop:prschar}, we obtain the desired symbol properties.\\
We now show that $\varphi^e$ and $\varphi^\psi$ may be obtained as the respective principal symbol components of a single $\SG$-phase function. For that we calculate the principal symbols of $\phi^e$ and $\phi^\psi$ by the means of the proof of Proposition \ref{prop:prschar}. Using $\lim_{n\rightarrow \infty}\tilde{y_n}\,\iota^{-1}(y_n)=\frac{y}{|y|}$ in case $y_n\rightarrow y$ with $y_n\in(\BBd)^o$ and $y\in\SSSd$ as well as \eqref{eq:yeypsi} and \eqref{eq:yeypsi2} in \eqref{eq:pte} and \eqref{eq:ptp} we obtain in the corner component
\begin{align*}
\sigma_\psi(\phi^e)|_{\SSSd\times\SSSd}&=(\id\times\id)^*\big\langle (y^\prime,y^\pp,y_d),\left(\eta_1,\eta^\prime,\tilde{H}^{\psi e}(y^\pp,\eta^\prime)\right)\big\rangle\\
\sigma_e(\phi^\psi)|_{\SSSd\times\SSSd}&=(\id\times\id)^*\left\langle y^\prime-\tilde{Y}^{\psi e}(y^\pp,\eta^\prime),\left(\eta_1,\eta^\prime\right)\right\rangle
\end{align*}
and thus we have 
\begin{multline*}
\sigma_\psi(\phi^e)|_{\SSSd\times\SSSd}-\sigma_e(\phi^\psi)|_{\SSSd\times\SSSd}=\\(\id\times\id)^*\left(\left\langle \tilde{Y}^{\psi e}(y^\pp,\eta^\prime),(\eta_1,\eta^\prime)\right\rangle+\left\langle (y^\pp,y_d),\tilde{H}^{\psi e}(y^\pp,\eta^\prime)\right\rangle\right),
\end{multline*}
which is nothing else than $\langle x,\, \xi\rangle$ restricted to $\SSSd\times\SSSd$ in $\Lmpe$ and thus vanishes by the conormality assumption. We are then able to, using \eqref{eq:yeypsi} and \eqref{eq:yeypsi2} and Proposition \ref{prop:prspairs-b}, continue $(\phi^e,\phi^\psi)$ to a single $\SG$-symbol with principal symbol $(\phi^e,\phi^\psi,\phi^{\psi e})$.\\
To have a chance of non-degeneracy, we first reduce the number of phase variables since, so far, the resulting phase function is constant in the $\xi^\pp$-variables. Getting rid of these redundant variables, we may define $\varphi:\RRd\times\RRs\rightarrow \RR$ by $((x^\prime,x^\pp);\theta)\mapsto\phi((x^\prime,x^\pp);(\theta,\xi^\pp_0))$ for some arbitrary $\xi^\pp_0$. We then obtain the components of the principal symbol $\varphi^\bullet=\sigma_\bullet\varphi$ for $\bullet\in\{e,\psi,\psi e\}$ and may define $\tilde{\varphi}\in\tilde{\gamma}^{-1}\tilde{y}^{-1}\Sm(\tilde{U})$ via $(\iota^{-1}\times\iota^{-1})^*\varphi$.\\
We now have to see that the functions $\varphi^\bullet$ indeed parametrize $\Lp$. For that we gather, by $\alpha^\bullet|_{\Lambda^\bullet}=0$, the identities
\begin{align*}
X^e(x^\pp,\xi^\prime)+\nabla_{\xi^\prime} \left(x^\pp\cdot \Xi^e(x^\pp,\xi^\prime)\right)&=0,\\
x^\pp\cdot \partial_{x^\pp_j}\Xi^e(x^\pp,\xi^\prime) &=0 \qquad j\in\{s+1,\dots,d\},\\
\theta\cdot \partial_{\xi^\prime_k} X^\psi(x^\pp,\xi^\prime)&=0 \qquad k\in\{1,\dots,s\}, \\
\nabla_{x^\pp}\left(\theta\cdot X^\psi(x^\pp,\xi^\prime)\right) + \Xi^\psi(x^\pp,\xi^\prime)&=0.
\end{align*}
We may then use these to compute, using \eqref{eq:phiedef} and \eqref{eq:phipsidef},
\begin{align*}
\nabla_\theta\varphi^e(x,\theta)&=x^\prime+\underbrace{x^\pp\cdot\nabla_\theta \Xi^e(x^\pp,\theta)}_{=-X^e(x^\pp,\theta)},\\
\partial_{\theta_k}\varphi^\psi(x,\theta)&=(x^\prime_k - X^\psi_k(x^\pp,\theta)) - \underbrace{\left(\partial_{\theta_k}X^\psi(x^\pp,\theta)\right)\cdot\theta}_{=0}.
\end{align*}
We therefore have $\nabla_\theta\varphi^\bullet=0$ if and only if $x^\prime=X^\bullet(x^\pp,\theta)$, and we have obtained 
\begin{align*}
\Cp^\bullet&=\{\big(X^\bullet(x^\pp,\theta),x^\pp;\theta\big)\}.
\end{align*}
In a similar fashion, using the remaining two identities,
\begin{align*}
	\Lp^\bullet&=
	\left\{\big(X^\bullet(x^\pp,\theta),x^\pp;\theta,\Xi^\bullet(x^\pp,\theta)\big)\right\}=
	\Lambda^\bullet.
\end{align*}
We can thus (locally) parametrize $\Lambda^\bullet$ by $\varphi^\bullet$. 
Finally, we have to check that the symbol $\varphi$ actually defines a phase function, meaning we have to check \eqref{eq:phaseineq}, which is equivalent to $|\nabla_x\varphi^\bullet|+|\nabla_\theta\varphi^\bullet|\neq 0$ on $\mathcal{B}^\bullet$. By assumption $\nabla_\theta$ vanishes only on $\Cp^\bullet$. There, however, we always have $\nabla_x\varphi^\bullet\neq 0,$ since by assumption none of the faces of $\Lp$ contains a point of the form $(x,0)$.\\
The proof is complete.
\end{proof}
Having established that we can always find a (local) parametrizing phase function for such an $\SG$-Lagrangian, we now investigate when two such phase functions may be considered \textit{equivalent}.
\begin{thm}
Let $\tilde{\varphi}_1,$ $\tilde{\varphi}_2\in\Sm(\BBd\times\BBs)$ be two non-degenerate phase functions that parametrize the same Lagrangian $\Lt\subset\WSG$ in a neighbourhood of $(y_0,\eta_0)\in\Lt$ such that
\begin{enumerate}
\item there exists $(y_0,\gamma_{0,1})\in\widetilde{\mathcal{C}}_{\varphi_1}$ and $(y_0,\gamma_{0,2})\in\widetilde{\mathcal{C}}_{\varphi_2}$ such that $(y_0,\eta_0)=\lpt_i(y_0,\gamma_{0,i})$ and $\tilde{\varphi}_1(y_0,\gamma_{0,1})=\tilde{\varphi}_2(y_0,\gamma_{0,2})$,\footnote{We note that this is always fulfilled in the classical case since, by homogeneity, $\varphi_i$ vanishes on $\mathcal{C}_{\varphi_i}$.}
\item The matrices $\left(\tilde{\gamma}^{-1}\tilde{y}\,\widetilde{\partial^2_{\theta_j\theta_k}\varphi_1}|_{X}\right)_{j,k=1,\dots, s}$ and
$\left(\tilde{\gamma}^{-1}\tilde{y}\,\widetilde{\partial^2_{\theta_j\theta_k}\varphi_2}|_{X}\right)_{j,k=1,\dots, s}$
have the same signature at $(y_0,\gamma_{0,i})\in\,\widetilde{\mathcal{C}}_{\varphi_i}$, where $\varphi_i:=(\iota\times\iota)^*\tilde{\varphi}_i$ are the (locally defined) phase functions associated with $\tilde{\varphi}_i$, $i=1,2$.\label{it:hypo2}
\end{enumerate}
Then, there exists a local homeomorphism $\tilde{\kappa}$ of the boundary $\Sct\mapsto\Sct$ that is defined in a neighbourhood of the $(y_0,\gamma_{0,2})$ in the corresponding faces, which is smooth on each face and such that $\tilde{\varphi_2}\circ\tilde{\kappa}=\tilde{\varphi}_1|_{\Sct}$.
\end{thm}
\begin{rem}
Note that the statement only ensures that the principal symbols of the corresponding phase functions $\varphi_i$ may be arranged to agree, that is the triples $(\varphi^e_i,\varphi^\psi_i,\varphi^{\psi e}_i)$. This is, however, not a drawback, since the principal symbols of $\varphi_i$ carry all the information about the asscociated sets of singularities $\Lpt$ and $\Cpt$, by Lemmas \ref{lem:Cpcl} and \ref{lem:Lpchar}.
\end{rem}
\begin{proof}
We assume $(y_0,\eta_0)\in\Ltpe$ since again this case (with slight adaptations) includes the others. Indeed, the case of $\Lptp$ is known from the classical theory and our proof follows the classical outline of  \cite{Hormander6} and \cite{Duistermaat}. We begin by arranging $\tilde{\varphi}_1$ and $\tilde{\varphi}_2$ such that they agree ``up to second order'' on $\widetilde{\mathcal{C}}_{\varphi_1}$. Consider the maps $\tilde{\Phi}_1$, $\tilde{\Phi}_2$ given by
\begin{align*}
(y,\gamma)\mapsto \tilde{\Phi}_i(y,\gamma):&=(\lpt_i,\tilde{y}\widetilde{\nabla_\theta\varphi_i})\in\BBd\times\BBd\times\RRd.
\end{align*}
By Theorem \ref{thm:repr} and Proposition \ref{prop:tl}, these maps are well-defined and smooth up to the boundary in a neighbourhood of $\widetilde{\mathcal{C}}_{\varphi_i}$. By Lemma \ref{lem:Cptchar} and Lemma \ref{lem:Lptchar} we have, for $(y,\gamma)\in\Sct$,
$$(\pr_3\circ\tilde{\Phi}_i)(y,\gamma)=0\Longleftrightarrow (y,\gamma)\in\widetilde{\mathcal{C}}_{\varphi_i}\Longleftrightarrow \tilde{\Phi}_i(y,\gamma)\in\Lt\times\{0\}.$$
By the implicit function theorem, that is Theorem \ref{thm:implicit}, and the non-degeneracy assumption of $\tilde{\varphi}_i$ we may thus locally invert in each face $\Sctp\cup\Sctpe=\BBd\times\SSS^{s-1}$ and $\Scte\cup\Sctpe=\SSSd\times\BBs$ separately, to obtain the two maps defined in neighbourhoods of $(y_0,\gamma_{0,i})$
\begin{align*}
\widetilde{\Psi}^\psi_i:(\Wt^\psi\cup{\Wt^{\psi e}})\times\RRd\rightarrow\BBd\times\SSS^{s-1},\\
\widetilde{\Psi}^e_i:(\Wt^e\cup{\Wt^{\psi e}})\times\RRd\rightarrow\SSSd\times\BBs.
\end{align*}
such that
$$\widetilde{\Psi}_i^\bullet\circ\left(\tilde{\Phi}_i|_{\Sc^\bullet}\right)=\id_{\Sct^\bullet},$$
i.e. we have the commuting diagram
\begin{center}
\begin{tikzpicture}
  \matrix (m) [matrix of math nodes,row sep=3em,column sep=2.2em,minimum width=2em]
  {
     (y_0,\gamma_{0,i}) & \widetilde{\mathcal{C}}_{\varphi_1}^\bullet\cup\widetilde{\mathcal{C}}_{\varphi_1}^{\psi e} & \Sc^\bullet\cup{\Sc^{\psi e}} & \Sc^\bullet\cup{\Sc^{\psi e}} \\
     (y_0,\eta_0) & \widetilde{\Lambda} & (\Wt^\bullet\cup{\Wt^{\psi e}})\times\RRd & (\Wt^\bullet\cup{\Wt^{\psi e}})\times\RRd\\};
    \node (m1112) at ($(m-1-1)!0.5!(m-1-2)$) {$\in$};
    \node (m1123) at ($(m-1-2)!0.5!(m-1-3)$) {$\subset$};
    \node (m1134) at ($(m-1-3)!0.5!(m-1-4)$) {$=$};
    \node (m2212) at ($(m-2-1)!0.5!(m-2-2)$) {$\in$};
    \node (m2223) at ($(m-2-2)!0.5!(m-2-3)$) {$\subset$};
    \node (m2234) at ($(m-2-3)!0.5!(m-2-4)$) {$=$};
	\draw[|->] (m-1-1) -- (m-2-1) node[midway,right] {$\widetilde{\Phi}_i$};
	\draw[->] (m-1-2) -- (m-2-2) node[midway,right] {$\widetilde{\Phi}_i$};
	\draw[->] (m-1-3) -- (m-2-3) node[midway,right] {$\widetilde{\Phi}_i$};
	\draw[<-] (m-1-4) -- (m-2-4) node[midway,right] {$\widetilde{\Psi}_i^\bullet$};	
\end{tikzpicture}
\end{center}
We may set, in a neighbourhood of $(y_0,\eta_0,0)$,
\begin{align*}
\widetilde{\Psi}^\psi_i|_{\Wtpe\times\RRd}=\widetilde{\Psi}^e|_{\Wtpe\times\RRd}=:\widetilde{\Psi}^{\psi e}.
\end{align*}
We also note that $\pr_1\circ\lpt_i=\id$. Therefore, the compositions $\widetilde{\Psi}^\bullet_1\circ \left(\tilde{\Phi}_2|_{\Sct^\bullet}\right)$ induce maps 
$$
\tilde{\kappa}^\bullet\colon
\widetilde{W}^\bullet\subseteq\widetilde{\mathcal{C}}_{\varphi_2}^\bullet\longrightarrow
\widetilde{\mathcal{C}}_{\varphi_1}^\bullet
\colon
(y,\gamma_2)\mapsto\big(y,\gamma_1(y,\gamma_2)\big),
$$
where $\widetilde{W}^\bullet$ is a neighbourhood of $(y_0,\gamma_0)$ in 
$\widetilde{\mathcal{C}}_{\varphi_2}^\bullet$. We thendefine 
$$
\tilde{\psi}:=\begin{cases}\tilde{\varphi}_2\circ\tilde{\kappa}^e & (y,\gamma)\in\Scte\\
\tilde{\varphi}_2\circ\tilde{\kappa}^\psi & (y,\gamma)\in\Sctp.
\end{cases}
$$
This yields a continuous function on the boundary $\Sct$ that is smooth in the interior of each boundary face up to the corner.
If we look at it as the principal symbol of a phase function, by means of Proposition \ref{prop:prschar}, we see that $\psi$ agrees (at the boundary) up to second order with $\varphi_1$ on $\widetilde{\mathcal{C}}_{\varphi_1}$, since their differentials vanish there (recall Lemmas \ref{lem:Cpcl} and \ref{lem:Lpchar}) and both functions are agree at the point $(y_0,\gamma_{0,1})$.\\
We can now essentially argue as in \cite{Hormander6} on each of the two faces. In fact, since all the objects involved are smooth up to the boundary of each face, Seeley's Extension Theorem allows us to extend them smoothly to a \textit{mirror copy} of $\widetilde{S}^\bullet$, across $\widetilde{S}^{\psi e}$. Such extensions, of course, still agree at $\widetilde{S}^{\psi e}$, and it is then possible to consider, for instance, Taylor expansions around points in $\widetilde{S}^{\psi e}$. To simplify the notation, in the sequel we omit the indication $e,\psi$ of the face, since the expressions will be well-defined on both faces. Let $\tilde{\varphi}$ and $\tilde{\psi}$ 
be two non-degenerate phase functions parametrizing the same Lagrangian and agreeing up to second order on $\widetilde{C}_\varphi=\widetilde{C}_\psi$, up to the boundary, in the sense above. Using this and 
the non-degeneracy of $\varphi$, setting $\tilde{h}_j=\tilde{y}\widetilde{\partial_{\theta_j}\varphi}(y,\gamma)$, 
$j=1,\dots,s$, we can write, at any given point in $\widetilde{C}_\varphi$,
\[
	\tilde{y}\tilde{\gamma}\tilde{\psi}(y,\gamma)=
	\tilde{y}\tilde{\gamma}\tilde{\varphi}(y,\gamma)+\sum_{j,k=1}^s\tilde{b}_{jk}(y,\gamma)
	\tilde{h}_j\tilde{h}_k,
\] 
for a symmetric matrix $\tilde{B}=(\tilde{b}_{jk}(y,\gamma))$. The non-degeneracy of $\tilde{\psi}$ is then
equivalent to
\[
	\det(I+\tilde{B} \tilde{A})\not=0 \text{ at $(y_0,\gamma_0)$},
\]
where we have set $\tilde{A}=\left(\tilde{\gamma}^{-1}\tilde{y}\,\widetilde{\partial^2_{\theta_j\theta_k}\varphi}(y,\gamma)\right)_{j,k=1,\dots,s}$. When $\tilde{B}$ is sufficiently small, we can show the equivalence between
$\tilde{\psi}$ and $\tilde{\varphi}$. In fact, by Taylor expansion,
\[
	\tilde{y}\tilde{\gamma}\tilde{\varphi}(y,\delta)=\tilde{y}\tilde{\gamma}\tilde{\varphi}(y,\gamma)+\sum_{j=1}^s
	(\delta_j-\gamma_j)\tilde{\gamma}\widetilde{\partial_{\theta_j}\varphi}(y,\gamma)+
	\sum_{j,k=1}^s\tilde{c}_{jk}(y,\gamma,\delta)(\delta_j-\gamma_j)\,(\delta_k-\gamma_k),
\]
with a symmetric matrix $\tilde{C}=(\tilde{c}_{jk})_{j,k=1,\dots,s}$. Setting
\[
	\delta_j=\gamma_j+\sum_{k=1}^s\tilde{w}_{jk}(y,\gamma)\,h_k,
\]
we prove the assertion if we show that there exist a matrix $\tilde{W}=(\tilde{w}_{j,k})_{j,k=1,\dots,s}$
such that
\[
	\tilde{W}+{^t}\tilde{W}\,\tilde{C}\,\tilde{W}=\tilde{B}.
\]
It is well known that, under the condition that the signatures of $\tilde{A}$ and $\tilde{C}$ agree, this equation has a solution for small $\tilde{B}$, which is in our cases implied by the hypothesis \textit{(\ref{it:hypo2})} and the fact that the two phase functions agree on $\Cpt$. The statement then follows, by determining a continuous family of non-degenerate phase functions $\tilde{\psi}_t$, $t\in[0,1]$, such that $\tilde{\psi}_0=\tilde{\varphi}$
and $\tilde{\psi}_1=\tilde{\psi}$. In fact, two elements $\tilde{\psi}_s$ and $\tilde{\psi}_t$ of such a family will be equivalent for $|s-t|$ sufficiently small. Since the procedure can be performed separately on the two faces,
and $\tilde{\psi}$ and $\tilde{\varphi}$ agree to second order up to the boundary including the corner, they are equivalent also there.\\
The details of this analysis, with reference to \cite{Hormander6}, are left to the reader.
\end{proof}

\appendix


\section{Manifolds with corners}
\label{subs:mfwc}
In this appendix we will present some results from the analysis on manifolds with corners that are employed in the study of $\SG$-Lagrangians. There are different definitions of manifold with corners, see \cite{Melrose4}, and, e.g. \cite{Jo,MO}.
Since in the main part of this document we only deal with finite-dimensional manifolds with corners, here we shortly recall the approach of \cite{MO} in such a case, while in its original formulation it is based on quadrants in general Banach spaces. Therein, the results needed for our purposes (notably, Theorem \ref{thm:linindepdiff} below) are explained in full detail, within the complete presentation of this theory.
\begin{defn}
	\label{def:qdr}
	With $d\in\mathbb{N}$, let $\Lambda\subseteq\{1,\dots,d\}$. The set
	\[
		E^+_{\Lambda,d}=
		\begin{cases}
			\RRd, &\text{if } \Lambda=\emptyset,
		 \\
			\{x\in \RRd\colon x_j\ge0, j\in\Lambda\}, &\text{otherwise},
		\end{cases}
	\]
	is called ($\Lambda$-)quadrant of $\RRd$. The notation $E^+_{j,d}$ is used
	when $\Lambda=\{j\}$. Obviously,
	\[
		E^+_{\Lambda,d}=\bigcap_{j\in\Lambda} E^+_{j,d}.
	\]
\end{defn}
The notion of differentiability on open subsets of a quadrant of $\RRd$ 
can be introduced exactly as on open subsets of $\RRd$.
\begin{defn}
	\label{def:diff}
	Let $U$ be an open subset of $E^+_{\Lambda,d}$, $f\colon U\to \RR^{d^\prime}$
	a map, and $x\in U$. Then, if there exists an element $u\in\cL(\RRd,\RR^{d^\prime})$ such that
	\[
		\lim_{y\to x}\frac{\|f(y)-f(x)-u(y-x)\|}{\|y-x\|}=0,	
	\]
	$\|.\|$ denoting the standard Euclidean norms on $\RRd$, $\RR^{d^\prime}$, 
	$f$ is said to be differentiable at $x$. In such a case, $u$ is called differential of $f$ at $x$ and is denoted by
	$Jf(x)$. If $f$ is differentiable at every $x\in U$, $f$ is said to be differentiable on $U$.
\end{defn}
The notion of differentiability and of differential in Definition \ref{def:diff} is well-defined and coincides with the ordinary one when
$\Lambda=\emptyset$. The basic properties and notions of differentiability, such as continuous differentiability and higher order differentiability, carry over to this notion of differentiation on quadrants. In particular, we call $f$ \textit{of class $\infty$}, or \textit{smooth (up to the boundary)} in a (relatively) open subset $U\subset \RRd$, denoted $f\in\Sm(U)$, if for every $p\in\NN$ the maps 
$J^pf:(\RRd)^{\otimes p}\rightarrow \RR^{d^\prime}$ are continuous and differentiable at every $x\in U$.

Equivalent alternative definitions of smooth maps on $E^+_{\Lambda,d}$ can be given in terms of existence of extensions on open sets of $\RRd$ including $U$, or on neighbourhoods in $\RRd$ of points $x\in U$, which are continuously differentiable of any order with respect to the standard notion, see \cite{MO}, Sections 1.1 and 2.1, for details.
\begin{defn}
	\label{def:atlas}
	Let $X$ be a set. The triple $(U,\nu,E^+_{\Lambda,d})$ is a chart on $X$ if:
	\begin{enumerate}
		\item $U\subseteq X$;
		\item $\nu\colon U\to E^+_{\Lambda,d}$ is an injective map and $\nu(U)$ is an open set of 
		$E^+_{\Lambda,d}$.
	\end{enumerate}
	Let $(U,\nu,E^+_{\Lambda,d})$, $(U^\prime,\nu^\prime,E^+_{\Lambda^\prime,d})$ be charts on $X$.
	They are smoothly compatible if $U\cap U^\prime=\emptyset$ or, if 
	$U\cap U^\prime\not=\emptyset$,
	\begin{enumerate}
		\setcounter{enumi}{2}
		\item $\nu(U\cap U^\prime)$ and $\nu^\prime(U\cap U^\prime)$ are
		open subsets of $E^+_{\Lambda,d}$ and $E^+_{\Lambda^\prime,d}$, respectively;
		\item $\nu^\prime\circ\nu^{-1}\colon\nu(U\cap U^\prime)\to\nu^\prime(U\cap U^\prime)$
		and $\nu\circ \nu^{\prime -1}\colon\nu^\prime(U\cap U^\prime)\to\nu(U\cap U^\prime)$
		are smooth maps.
	\end{enumerate}
	A collection $\cA$ of smoothly compatible charts that cover $X$ is called a \textit{smooth atlas}. As usual, two atlases	
	$\cA$, $\cA^\prime$ are called 
	equivalent if $\cA\cup\cA^\prime$ is an atlas, which yields an equivalence relation. An equivalence class $[\cA]_{\sim}$ 
	is called \textit{smooth differentiable structure on $X$} and the pair 
	$(X,[\cA]_{\sim})$ is called smooth manifold or a $\Sm$-manifold, 
	denoted simply by $X$. If $\Lambda$ cannot be chosen as empty, $X$ is called a smooth manifold with corners.
\end{defn}
Given a $\Sm$-manifold $X$, the set $\{U\subseteq X\colon \text{ $U$ is the domain of a chart on $X$}\}$ is a basis for a topology on $X$. The space of smooth maps among $\Sm$-manifolds $X$ and $Y$, denoted by $\Sm(X,Y)$, is defined in a completely similar fashion to the usual way. In particular the tangent bundle may be defined in a neighbourhood $U$ given by a chart as $U\times \RRd$, and consequently over the full manifold by imposing contravariant transformation behaviour. The differential of a smooth map $f:X\rightarrow Y$ in local coordinates then induces a map $df:TX\rightarrow TY$.
\begin{defn}
	\label{def:indx}
	Let $U$ be an open set of $E^+_{\Lambda,d}$.
	\begin{enumerate}
	\item For $x\in E^+_{\Lambda,d}$, $\mathrm{ind}(x):=\mathrm{ind}_\Lambda(x)=\#\{j\in\Lambda \colon x_j=0\}$;
	\item The set $\{x\in U\colon \mathrm{ind}(x)\ge1\}$ is called boundary of $U$, 
	and denoted $\partial_\Lambda U=\partial U$;
	\item The set $\{x\in U\colon \mathrm{ind}(x)=0\}$ is called interior of $U$,
	and denoted $\mathrm{int}_\Lambda U =\mathrm{int} \, U = U^o$.
	\end{enumerate}
\end{defn}
It can be proved that the value $\mathrm{ind}(x)$ is invariant under smooth diffeomorphisms\footnote{A smooth diffeomorphism is a smooth bijective map $X\rightarrow X$ whose inverse is also smooth.}, that is,
it has an invariant meaning on a manifold $X$. This implies that also the notions of boundary and interior
are invariantly defined on $X$. More generally, for any $k\in\NN$, it is possible to define $\partial^kX$,
the $k$-boundary of $X$, as the set of all points $x\in X$ such that $\mathrm{ind}(x)\ge k$. We set
$\partial X:=\partial^1X$. Moreover, for any $k\in\NN$, the set 
$\{x\in X\colon \mathrm{ind}(x)=k\}$ is denoted by $B_kX$. The set $B_0X$
is called the interior of $X$.
\begin{ex}
\label{ex:ball}%
Consider $d\in\NNz$, $\BB^d=\{y\in\RRd\colon\|y\|\le1\}$, and, for all $j\in\{1,\dots,d\}$,
$(V^+_j,\nu^+_j,E^+_{j,d})$, $(V^-_j,\nu^-_j,E^+_{j,d})$, where
\begin{itemize}
	\item $V^+_j=\{y\in\BBd\colon y_j>0\}$, $V^-_j=\{y\in\BBd\colon y_j<0\}$;
	\item $\displaystyle\nu^+_j(y)=(\dots,y_{j-1},\sqrt{1-(\dots+y_{j-1}^2+y_{j+1}^2+\dots)}-y_j,y_{j+1},\dots)$;
	\item $\displaystyle\nu^-_j(y)=(\dots,y_{j-1},\sqrt{1-(\dots+y_{j-1}^2+y_{j+1}^2+\dots)}+y_j,y_{j+1},\dots)$.
\end{itemize}
Then, it turns out that 
\[
\cA=\{(V^+_j,\nu^+_j,E^+_{j,d})\}_{j=1,\dots,n}
\cup\{(V^-_j,\nu^-_j,E^+_{j,d})\}_{j=1,\dots,n}\cup\{(\BBd)^o,\mathrm{id},\RRd)\}
\]
is a smooth atlas on $\BB^n$. Furthermore, the topology of of the manifold $(\BBd,[\cA])$ is
the usual (subset) topology of $\BBd\subset\RRd$, $\partial\BBd=\SSS^{n-1}$, $\partial^2\BBd=\emptyset$. 
\end{ex}
\begin{prop}
	\label{prop:diffeodx}
	Let $X$, $X^\prime$ be $\Sm$-manifolds,  $f	\colon X\to X^\prime$ 
	a diffeomorphism. Then, for any $k\in\NNz$, $f(\partial^kX)=\partial^kX^\prime$. 
	In particular, if $\partial^2X=\emptyset$, $f$ is a diffeomorphism of $\partial X$
	onto $\partial X^\prime$.
\end{prop}
It is well known that the finite Cartesian product of manifolds without boundary is a natural, well-defined
construction, which yields another manifold without boundary. However, in the category of manifolds with
boundary (i.e., $\partial^2X=\emptyset$), there is no such a natural finite product construction. It turns out
that the category of manifolds with corners is the suitable one in which to define finite Cartesian products.
\begin{prop}
	\label{prop:prd}
	Let $X,X^\prime$ be $\Sm$-manifolds. Then, there exists a unique $\Sm$-structure $[\cA]$ on $X\times Y$ such that, 
	for every chart $(U,\nu,E^+_{\Lambda,d})$
	on $X$ and every chart $(U^\prime,\nu^\prime,E^+_{\Lambda^\prime,d^\prime})$ 
	on $X^\prime$, $(U\times U^\prime,\nu\times\nu^\prime,E^+_{\Lambda\amalg\Lambda^\prime,d+d^\prime})$,
	$\Lambda\amalg\Lambda^\prime=\Lambda\cup\{d+j^\prime\colon j^\prime\in\Lambda^\prime\}$, 
	is a chart of $(X\times X^\prime,[\cA])$. The pair $(X\times X^\prime,[\cA])$ is called
	the product manifold of $X$ and $X^\prime$.
\end{prop}
\begin{prop}
	\label{prop:prdbis}
	Let $X,X^\prime$ be $\Sm$-manifolds. Then, the following statements hold true.
	\begin{enumerate}
		\item The topology of the product manifold $X\times X^\prime$ is the product topology
		of those on $X$ and $X^\prime$.
		\item For every $(x,x^\prime)\in X\times X^\prime$, 
		$\mathrm{ind}(x,x^\prime)=\mathrm{ind}(x)+\mathrm{ind}(x^\prime)$.
		\item For all $l\in\NNz$, $\displaystyle\partial^l(X\times X^\prime)=\bigcup_{\substack{j+k=l\\j,k\ge0}}
		\partial^jX\times\partial^k X^\prime$. Moreover, $(X\times X^\prime)^o=X^o\times (X^\prime)^o$.
	\end{enumerate}	
\end{prop}
\begin{ex}
This proposition allows us to construct a differential structure on $\BBd\times\BBs$, $s\in\NNz$, in terms of that in Example \ref{ex:ball}, that turns this set into a manifold with corners of codimension $2$ such that 
$$B_k(\BBd\times\BBs)=
\begin{cases}
(\BBd)^o\times(\BBs)^o & k=0\\
((\BBd)^o\times\SSS^{s-1}\big)\cup\big(\SSSd\times(\BBs)^o\big) & k=1\\
\SSSd\times\SSS^{s-1} & k=2 \\
\emptyset & k>2.
\end{cases}
$$
\end{ex}
It is a remarkable aspect of this theory that the implicit function theorem extends to manifolds with corners,
under a rather mild (and natural) additional condition on boundaries. In the next statement, given a map
$f\colon X\times Y\to Z$, for any $(a,b)\in X\times Y$,
we write $d_{(a,b)}f=(d^X_{(a,b)}f,d^Y_{(a,b)}f)$ with the linear morphisms
$d^X_{(a,b)}f\colon T_aX\to T_{f(a,b)}Z$ and $d^Y_{(a,b)}f\colon T_bY\to T_{f(a,b)}Z$.
\begin{thm}
	\label{thm:implicit}
	Let $X,Y,Z$ be $\Sm$-manifolds, $f\colon X\times Y\to Z$
	a smooth map and $(a,b)\in X\times Y$. Suppose that $d^Y_{(a,b)}f\colon T_bY\to T_{f(a,b)}Z$
	is a linear homeomorphism, and suppose that there are open neighbourhoods $V_a$ of $a$ and
	$V_b$ of $b$ such that $f(V_a\times(V_b\cap\partial Y))\subset \partial Z$.\\
	Then there exist an open neighborhood $W_a$ of $a$, an open neighbourhood $W_b$ of $b$ and a 
	unique map $g\colon W_a\to W_b$ such that $f(x,g(x))=f(a,b)$ for $x\in W_a$. Furthermore:
	\begin{enumerate}
		\item $g(a)=b$, and $g$ is smooth on $W_a$;
		\item for every $x\in W_a$, $d^Y_{(x,g(x))}f$ is a linear homeomorphism and
		\[
			d_xg=-(d^Y_{(x,g(x))}f)^{-1}\circ d^X_{(x,g(x))}f.
		\]
	\end{enumerate}
\end{thm}
We now state the definition of a submanifold (with corners) in this setting.
\begin{defn}
	\label{def:subm}
	Let $X$ be a $\Sm$-manifold and $X^\prime\subset X$. Then, $X^\prime$ is a 
	$\Sm$-submanifold of $X$ if, for every $x^\prime\in X^\prime$, there exist:
	\begin{enumerate}
		\item a chart $(U,\nu,E^+_{\Lambda,d})$ of $X$ such that $x^\prime\in U$ and $\nu(x^\prime)=0$;
		\item an integer $d^\prime\in\mathbb{N}$, $d^\prime \le d$, and $\Lambda^\prime\subseteq
		\{1, \dots, d^\prime\}$, such that
		$\nu(U\cap X^\prime)=\nu(U)\cap E^{+}_{\Lambda^\prime, d^\prime}$, and 
		$\nu(U)\cap E^{+}_{\Lambda^\prime, d^\prime}$ is an open subset of $E^{+}_{\Lambda^\prime, d^\prime}$.
	\end{enumerate}
\end{defn}
In particular, $X^o$ is an open submanifold of $X$ and if $\partial^2X=\emptyset$, $\partial X$
is a submanifold of $X$. In general, there is no relation between the boundary of $X$ and that of a submanifold
of $X$. This leads to the definition of special submanifolds, whose boundaries have ``good positions'' within the boundary of the ambient manifold.
\begin{defn}
	\label{def:totneat}
	Let $X^\prime$ be a submanifold of $X$. Then:
	\begin{enumerate}
		\item $X^\prime$ is a \textit{neat submanifold} of $X$ if $\partial X^\prime=(\partial X)\cap X^\prime$;
		\item  $X^\prime$ is a \textit{totally neat submanifold} of $X$ if, for all $x^\prime\in X^\prime$,
		$\mathrm{ind}_{X^\prime}(x^\prime)=\mathrm{ind}_X(x^\prime)$, that is,
		 $B_k X^\prime=X^\prime\cap B_kX$ for any $k\in\NN$.
	\end{enumerate}
\end{defn}
An equivalent condition for $X^\prime$ to be a totally neat submanifold of $X$ is that, for all $x^\prime\in
X^\prime\cap B_kX$,
\[
	\partial X^\prime=(\partial X)\cap X^\prime
	\text{ and }
	T_{x^\prime}X=(d_{x^\prime}j^\prime)(T_{x^\prime}X^\prime)+(d_{x^\prime}j)(T_{x^\prime}B_kX),
\]
where $j^\prime\colon X^\prime\hookrightarrow X$ and $j\colon B_kX\hookrightarrow X$ are the canonical
inclusions. The properties of being a neat or totally neat submanifold are invariant under diffeomorphisms.
\begin{defn}
	\label{def:imm}
	Let $f\colon X\to X^\prime$ be a $\Sm$-map and $x\in X$. $f$ is called (smooth) immersion at $x$
	if there is a chart $(U,\nu,E^+_{\Lambda,d})$ on $X$ such that $\nu(x)=0$, and a chart
	$(U^\prime,\nu^\prime,E^+_{\Lambda^\prime,d^\prime})$ on $X^\prime$ with $\nu^\prime(f(x))=0$,
	such that $f(U)\subseteq U^\prime$, 
	$\nu(U)\subset\nu^\prime(U^\prime)$ and $\nu^\prime\circ
	f_{|U}\circ\nu^{-1}\colon \nu(U)\to\nu(U^\prime)$ is the inclusion map. If $f$ is an immersion $\forall x\in X$, it is called immersion on $X$.
\end{defn}
\begin{thm}
	\label{thm:imm}
	Let $f\colon X\to X^\prime$ be a smooth map and $x\in X$ 
	such that $f(x)\in(X^\prime)^o$. Then, the following
	statements are equivalent:
	\begin{enumerate}
		\item $f$ is an immersion at $x$;
		\item $d_xf$ is an injective map.
	\end{enumerate}
\end{thm}
We now recall the definition of embeddings in this context, and describe how they can be
characterized.
\begin{defn}
	\label{def:emb}
	Let $f\colon X\to X^\prime$ be a map of class $p$. Then, $f$ is called embedding
	if it is an immersion and $f\colon X\to f(X)$ is a homeomorphism.
\end{defn}
We may now give a characterization of embedded submanifolds.
\begin{prop}
	\label{prop:embchar}
	Let $X,X^\prime$ be $\Sm$-manifolds and $f\colon X\to X^\prime$ a map.
	Then, the following statements are equivalent:
	\begin{enumerate}
		\item $f$ is a smooth embedding;
		\item $f(X)$ is a $\Sm$-submanifold of $X^\prime$ and $f\colon X\to
		f(X)$ is a diffeomorphism.
	\end{enumerate}
\end{prop}
The next result, \cite[Prop. 4.2.10]{MO}, with which we conclude this appendix, 
shows that also on manifolds with corners the 
solutions to systems of equations give rise to submanifolds, provided that the corresponding 
differentials are linearly independent.
\begin{thm}
	\label{thm:linindepdiff}
	Let $X$ be a smooth manifold and $f_1, \dots, f_s\colon X\to \RR$ be $\Sm(X)$-maps. Consider the set $Y=\{x\in X\colon f_1(x)=\dots=f_s(x)=0\}$, and suppose that, for every 
	$x\in Y$, $(d_x(f_1|_{B_kX}),\dots, d_x(f_s|_{B_kX}))$ is a linearly independent system of elements of
	$(T_x(B_kX))^*$, where $k=\mathrm{ind}(x)$. Then we have
	\begin{enumerate}
		\item $Y$ is a closed totally neat $\Sm$-submanifold of $X$;
		\item $T_x(j)(T_xY)=\{v\in T_xX\colon T_xf_1(v)=\dots=T_xf_n(v)=0\}$, where $j\colon Y\to X$ is
		the inclusion map and $x\in Y$;
		\item For all $x\in Y$, $\mathrm{codim}_x Y=s$.
	\end{enumerate}
\end{thm}

\bibliographystyle{amsalpha}

%

\address{Dipartimento di Matematica ``G. Peano''\newline
\indent
Universit\`a degli Studi di Torino\newline
\indent
V. C. Alberto, n. 10, I-10126 Torino, Italy\\}
\email{\indent sandro.coriasco@unito.it\\ \smallskip}

\address{\indent Institut f\"ur Analysis\newline
\indent
Leibniz Universit\"at Hannover \newline
\indent
Welfengarten 1, D-30167 Hannover, Deutschland\\}
\email{\indent rschulz@math.uni-hannover.de}

\end{document}